\documentclass[a4paper,11pt]{article}
\usepackage{graphicx,amsmath,amsfonts,amssymb,amsthm,color,framed,dsfont,hyperref,accents, pgfplots}
\usepackage[T1]{fontenc} 
\usepackage[french,english]{babel}

\usepackage{geometry}
\pgfplotsset{width=10cm,compat=1.9}
\usetikzlibrary{external}
\usepackage{physics}
\usepackage{bbm}
\usepackage[font={small,it}]{caption}

\newcommand{\R}{\mathbb{R}}

\newcommand{\Cl}{\mathsf{Cl}}

\newcommand{\supp}{\mathrm{supp}} 

\newcommand{\Prob}{\mathbb{P}}

\newcommand{\TrProb}{\mathsf{P}}

\newcommand{\E}{\mathbb{E}}

\newcommand{\1}{\mathds{1}} 
\newcommand{\Wass}{\mathbb{W}}
\newcommand{\Law}{\mathcal{L}}

\newcommand{\e}{\mathrm{e}}

\newcommand{\ds}{\displaystyle}
\newcommand{\cMc}{\mathcal{M}}
\newcommand{\cGc}{\mathcal{G}}

\newcommand{\cDc}{\mathcal{D}}
\newcommand{\cCc}{\mathcal{C}}
\newcommand{\cHc}{\mathcal{H}}

\newcommand{\poin}[1]{\dot{#1}}

\newtheorem{thm}{Theorem}[section]

\newtheorem{cor}[thm]{Corollary}

\newtheorem{lem}[thm]{Lemma}
\newtheorem{prop}[thm]{Proposition}

\newtheorem{defn}[thm]{Definition}

\newtheorem{rem}[thm]{Remark}
\numberwithin{equation}{section}

\newtheorem{innercustomthm}{Assumption}
\newenvironment{custom_assumption}[1]
  {\renewcommand\theinnercustomthm{#1}\innercustomthm}
  {\endinnercustomthm}

\title{Measure-dependent non-linear diffusions with superlinear drifts: asymptotic behaviour of the first exit-times}

\date{\today}

\author{A.~Aleksian$^{1,a}$ and J.~Tugaut$^{1,b}$\\[5pt]
\small{$^1$ Université Jean Monnet, CNRS UMR 5208,}\\
\small{Institut Camille Jordan, Maison de l'Université, 10 rue Tréfilerie,}\\
\small{CS 82301, 42023 Saint-Étienne Cedex 2, France}\\
\small{$^a$ ashot.aleksian@univ-st-etienne.fr, $^b$ julian.tugaut@univ-st-etienne.fr}
}

\begin{document}

%
%

\maketitle

\begin{abstract}
In this paper, we study McKean-Vlasov SDE living in $\R^d$ in the reversible case without assuming any type of convexity assumptions for confinement or interaction potentials. Kramers' type law for the exit-time from a domain of attraction is established. Namely, in the small-noise regime, the limit in probability of the first exit-time behaves exponentially. This result is established using the large deviations principle as well as improved coupling method.

Having removed the convexity assumption, this work is a major improvement of the previously known results for the exit-time problem, the review of which is provided in the paper. 
\end{abstract}

 \par

 \par

\section{Introduction}\label{s:Introduction}


Let us consider $(X_t^\sigma, \, \, t\ge 0)$ a   measure-dependent stochastic process (also called McKean-Vlasov diffusion \cite{McKean1,McKean2}), solution of the following stochastic differential equation (SDE):
\begin{equation}
    \label{eq:main_SDE}
    \dd X_t^\sigma = \sigma \dd B_t - \nabla V(X_t^\sigma) \dd{t} - \nabla F\ast\mu_t^\sigma (X_t^\sigma) \dd{t},\quad X_0^\sigma = x_{\mathrm{init}} \in \R^d.
\end{equation}
Here $(B_t,\, t\ge 0)$ stands for the $d$-dimensional Brownian motion, $V$ represents the environment which is assumed to be a multi-well function (also called confinement potential in this work) and $F$ is the interaction potential corresponding to the form and strength of interaction of the process with its law. This specific form of the McKean-Vlasov diffusion is also known in the literature under the name of self-stabilizing diffusion or SSD (see \cite{HIP2}).

The aim of this study is to describe how long the stochastic process stays in a domain $\cDc$, which is a neighborhood of a local minimum of $V$, before its first exit from this neighborhood. Therefore, the main object of interest in this paper is the following stopping time:

\begin{equation}
    \label{eq:def:tau}
    \tau_\cDc^\sigma:=\inf\{t\ge 0:\ X_t^\sigma \notin \cDc  \}.
\end{equation}
The precise assumptions under consideration are given later on.












\subsection{Organization of the paper}\label{s:Organization}

The current section is followed by presenting and discussing assumptions on potentials $V$ and $F$ and domain $\cDc$, exit-time from which is considered. An existing result on existence and uniqueness of the process under almost identical assumptions are provided in Section~\ref{ss:Existence} with a discussion on how its proof can be adapted to our case. The large deviations principle for this system is provided in Section~\ref{ss:LDP}.

Main results of this paper are formulated in Section~\ref{s:Main_results}. Namely, the Kramers' type law for exit-time and the exit-location results for both cases of bounded and unbounded domain $\cDc$. This theorem is followed by Section~\ref{ss:comparison} comparing them to previously known results for exit-time problem in the case of self-stabilizing diffusion and Section~\ref{s:extensions} discussing open questions and possible extensions of our findings.

Section~\ref{s:interm_results} contains intermediate lemmas that are necessary for the proof of the main theorem of the paper. These lemmas are proved in Section~\ref{s:interm_res_proof}.  
Section~\ref{s:main_res_proof} contains the proof of the main theorem provided in Section~\ref{s:Main_results}.

\subsection{Assumptions}\label{s:Assumptions}

Here, we give the assumptions on the potentials and on the domain. 

\begin{custom_assumption}{A-1}\label{assu:pot:V}
    Let us consider the following hypotheses concerning the confinement potential:
    \begin{description}
        \item[$(V-1)$] The confinement potential is a regular function $V \in \mathcal{C}^2(\R^d)$.
        \item[$(V-2)$] $V$ is uniformly convex at infinity. Namely, there exists $\theta_1 > 0$ and $R > 0$ such that for all $x \in \R^d$ satisfying $|x| > R$ we have $$\nabla^2 V(x) \succeq \theta_1 {\rm Id},$$ where ${\rm Id}$ is the identity matrix.
        \item[$(V-3)$] There exist $r \in \mathbb{Z}_+$ and a constant $C>0$ such that 
        \[
        | \nabla V(x)|\le C(1+ | x|^{2r-1}), \quad\mbox{for all}\quad x\in\mathbb{R}^d.
        \]
        \item[$(V-4)$] There exists $a \in \mathbb{R}^d$ such that $\nabla V(a)=0$ and $\nabla^2V(a) \succeq \rho_1{\rm Id}$ for some $\rho_1>0$, where ${\rm Id}$ is the identity matrix.
        \item[$(V-5)$] The function $\nabla V$ is locally Lipschitz. More precisely, for any $x\in\R^d$ and $y\in\R^d$, we have:
        \begin{equation}\label{eq:rem:locallip}
            | \nabla V(x)-\nabla V(y)|\le C|x-y|(1+ | x|^{2r-1}+| y|^{2r-1})\,,
        \end{equation}
        where $r$ has been introduced in $(V-3)$.
    \end{description}
\end{custom_assumption}

Assumption $(V-1)$ is natural since we will use Itô calculus to obtain some of our results. Thus, we require that $V$ is of class $\mathcal{C}^2$. Assumption $(V-2)$ is taken to ensure that the confinement potential forces the diffusion to stay in a compact set and thus that the process does not explode. Assumptions $(V-3)$ and $(V-5)$ are required to comply with the theory developed in \cite{BRTV} for ensuring the existence of the self-stabilizing diffusion when the drift is superlinear. Assumption $(V-4)$ means that there is a local minimizer with a non-degenerate Hessian. We point out that $\nabla V$ is not assumed to be globally Lipschitz.



Assumption \ref{assu:pot:V} covers a wide range of possible multi-well potentials. An analytical example of such a potential $V$ that satisfies Assumption~\ref{assu:pot:V} in dimension $d = 1$ could be the classical double-well potential (see Fig.~\ref{fig:Example_V_d1})
\[
V(x):=\frac{x^4}{4} - \frac{x^2}{2}.
\]

In dimension two, the following function
\begin{align*}
    V(x_1, x_2)&:= \frac{3}{2}\left(1 - x_1^2 - x_2^2\right)^2+\frac{1}{3}\left(x_1^2 - 2\right)^2+\frac{1}{6}\left((x_1 + x_2)^2 - 1\right)^2\\
    &\quad +\frac{1}{6}\left((x_1 - x_2)^2-1\right)^2\,
\end{align*}
could be an example of a double-well potential satisfying these assumptions. Fig.~\ref{fig:Example_V_d2} shows its level sets.

We now give the assumptions on the interaction potential.


\begin{figure}[h]
\centering
\begin{minipage}[b]{.5\textwidth}
  \centering
    \includegraphics{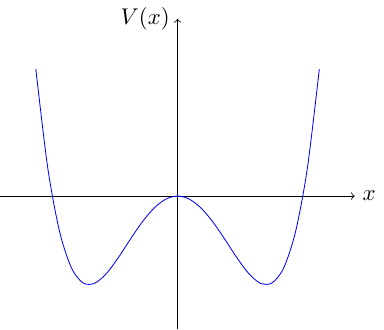}
  \caption{Example of $V$ in dimension $d = 1$.}
  \label{fig:Example_V_d1}
\end{minipage}%
\begin{minipage}[b]{.5\textwidth}
  \centering
    \includegraphics{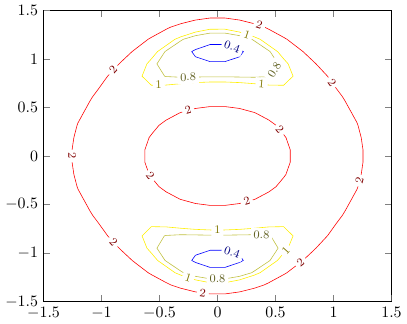}
  \caption{Example of $V$ in dimension $d = 2$.}
  \label{fig:Example_V_d2}
\end{minipage}
\end{figure}

\begin{custom_assumption}{A-2} \label{assu:pot:F}
Let $\theta_1$ and $r$ be the positive constants introduced in $(V-2)$ and $(V-3)$. Consider the following hypotheses concerning the interaction:
\begin{description}
    \item[$(F-1)$] The interaction potential is a regular function $F \in \cCc^2(\mathbb{R}^d)$.
    \item[$(F-2)$] $F(0) = 0$ and $\nabla F$ is rotationally invariant, that is there exists a continuous function $\phi : [0; \infty ) \to \R$ with $\phi(0) = 0$ such that
    \begin{equation*}
        \nabla F(x) = \frac{x}{|x|} \phi(|x|).
    \end{equation*}
    \item[$(F-3)$] There exists a constant $C'>0$ such that 
    \[
    | \nabla F(x)|\le C'(1+ | x|^{2r-1}), \quad\mbox{for all}\quad x\in\mathbb{R}^d\,.
    \]
    \item[$(F-4)$] The function $\nabla F$ is locally Lipschitz. More precisely, for any $x\in\R^d$ and $y\in\R^d$, we have:
    \begin{equation}
        \label{eq:rem:locallip2}
        | \nabla F(x)-\nabla F(y)|\le C'|x-y|(1+ | x|^{2r-1}+| y|^{2r-1})\,.
    \end{equation}
    \item[$(F-5)$] There exists a constant $\theta_2 > 0$ such that for any $x \in \R^d$ we have $$\nabla^2 F(x) \succeq -\theta_2 {\rm Id},$$ where ${\rm Id}$ is the identity matix. Moreover, $\theta_1 > \theta_2$.
\end{description}
\end{custom_assumption}

Again, Assumption $(F-1)$ is natural since we will use Itô calculus. Assumption $(F-2)$ is taken to ensure existence and uniqueness of the process following the work \cite{HIP2}, where a similar assumption was introduced. We point out that the exact value of $F(0)$ does not have any effect on our methods, however, taking it equal to $0$ simplifies the writing. Note that we do not use assumption $(F-2)$ for proving the exit-time result. Assumption $(F-3)$ is required for using the method developed in \cite{BRTV, HIP2} about the existence of the self-stabilizing diffusion when the drift is superlinear. We point out that $\nabla F$ is not assumed to be globally Lipschitz. Assumption $(F-5)$ is taken in order to guarantee that the attractive behaviour at infinity of $V$ will not be overcome by $F$, which is essential for existence and uniqueness results (we provide this result in Section~\ref{ss:Existence}).





Assumption \ref{assu:pot:F} covers a wide range of possible interaction potentials defining various behaviour with respect to the law of the process (attractive, repulsive or the combination of two). A classical analytical example of the interaction potential in general dimension $d$ is
\[
    F(x) := \pm\frac{\alpha}{2}|x|^2\,,
\]
with $\alpha > 0$. In the case of $F(x) = \frac{\alpha}{2}|x|^2$ (see Fig.~\ref{fig:Example_F_1} for its depiction in $d = 1$), the interacting potential is globally convex and induces attracting behaviour, whereas it is globally concave and thus repulsive with the negative sign. Another possible example of a potential is
\[
    F(x) := C\e^{-\frac{\theta}{2}|x|^2}\,,
\]
with $\theta > 0$ (for its graph in $d = 1$ see Fig.~\ref{fig:Example_F_2}). In this case, the function is neither convex nor concave, but, after a careful examination, we can see that it still exhibits repulsive behaviour, though dissipating at infinity. Note that here, despite assumption $(F-2)$, $F(0) \neq 0$. As was pointed out above, the translations of $F$ do not influence the dynamic of \eqref{eq:main_SDE}.


\begin{figure}[t]
\centering
\begin{minipage}[b]{.5\textwidth}
  \centering
    \includegraphics[]{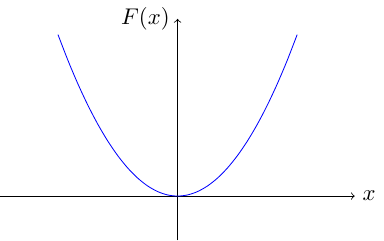}
  \captionof{figure}{Example of a convex $F$.}
  \label{fig:Example_F_1}
\end{minipage}%
\begin{minipage}[b]{.5\textwidth}
  \centering
    \includegraphics[]{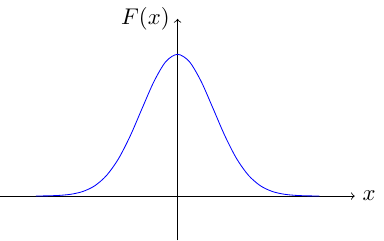}
  \captionof{figure}{Example of a non-convex $F$.}
  \label{fig:Example_F_2}
\end{minipage}
\end{figure}

In the following, we introduce the assumptions on the domain. First, we define the effective (in the small-noise limit) potential.

\begin{defn}
\label{def:effectivepotential}
Let $a$ be the local minimizer of $V$ introduced in \ref{assu:pot:V}. Then, $W_a \in \mathcal{C}^2(\R^d)$ such that $W_a := V + F \ast\delta_a = V + F(\cdot - a)$ is called the effective potential.
\end{defn}

The name ``effective'' comes from the fact that, as will be proved below, before the exit-time from the stable domain $\cDc$, for small $\sigma$, the potential $V + F \ast \mu^\sigma_t$, inducing the drift term of our process, is well approximated by $W_a$.

In order to ensure that, in the small-noise limit, our process behaves well around the attractor $a$, we need to assume that $a$ is also a stable local minimizer of the effective potential. Consider the following assumption

\begin{custom_assumption}{A-4}
\label{assu:pot:VetF2}
    The matrix $\nabla^2 W_a(a) = \nabla^2V(a)+\nabla^2F(0)$ is positive definite.
\end{custom_assumption}

Note that Assumption \ref{assu:pot:VetF2}, along with the continuity assumptions on $\nabla^2 V$ and $\nabla^2 F$ (Assumptions \ref{assu:pot:V} and \ref{assu:pot:F}), leads to the fact that we can find an open neighborhood of the point $a$ such that $W_a$ is convex inside it. Consider:
\begin{defn}\label{def:Loc_convexity}
    Let $\rho > 0$ be a small enough positive number such that $W_a$ is convex inside $B_\rho(a)$. Let $C_{W} > 0$ be a constant such that for any $x \in B_{\rho}(a)$:
    \begin{equation*}
        \nabla^2 W_a(x) = \nabla^2 V(x) + \nabla^2 F(x-a) \succeq C_{W}{\rm Id}\,,
    \end{equation*}
    where ${\rm Id}$ is the identity matrix.
\end{defn}

Let us now introduce assumptions regarding the domain of interest $\cDc \subset \R^d$, exit-time from which will be considered in the future. First assumption on domain $\cDc$ is the following:

\begin{custom_assumption}{A-5}
\label{assu:domain:bounded}
    $\cDc$ is a bounded connected open subset of $\mathbb{R}^d$ containing the point $a$.
\end{custom_assumption}

\begin{rem}
    Without loss of generality, we choose $\rho > 0$ from Definition~\ref{def:Loc_convexity} to be small enough such that we have the following  strict inclusion $B_\rho(a) \subset \cDc$.
\end{rem}

The boundedness of the domain $\cDc$ will be relaxed later. However, the fact that $\cDc$ is connected and open is mandatory and classical from \cite{DZ,FW}.

The following assumptions on $\cDc$ are mandatory:

\begin{custom_assumption}{A-6}
\label{assu:firsttraj}
    The domain $\cDc$ contains the deterministic path $(\gamma_t,\, t\ge 0)$ solution of the following dynamical system
    \begin{equation}
    \label{eq:dyn}
        \frac{\dd}{\dd{t}}\gamma_t =  -\nabla V(\gamma_t),\quad\quad \gamma_0 = x_\mathrm{init}.
    \end{equation}
    We assume furthermore that $\lim_{t\to\infty}\gamma_t = a$.
\end{custom_assumption}

This assumption is important for the type of exit-problem that we consider here, which is exit created by the small noise from a domain of attraction. We will see further, using the large deviations principle (LDP), that for any $T > 0$, the processes $(X_t^\sigma, 0\leq t\leq T)$ and $(\gamma_t,0\leq t\leq T)$ are close in supremum norm with high probability when $\sigma$ is small enough. In the case where $T_1 := \inf\{t \geq 0: \gamma_t \notin \cDc\} < \infty$, it is easy to show, using LDP, that $\tau_\cDc^\sigma \approx T_1$ for small $\sigma$. In other words, it is impossible to obtain the Kramers' type law without Assumption~\ref{assu:firsttraj}.

Now, we present the definition of a stable domain.

\begin{defn}
\label{def:stable}
We say that an open connected subset $\cGc$ of $\R^d$ is stable by the vector field $-\nabla W_a$ if for any $t\geq0$, for any $x\in\cGc$, $\psi_t(x)\in\cGc$ where the process $\psi(x)$ is the solution to the following dynamical system:
\[
    \psi_t(x) = x - \int_0^t \nabla W_a(\psi_s(x))\dd{s}.
\]

\end{defn}

This leads to a classical assumptions on the domain $\cDc$ that is standard for the Freidlin-Wentzell theory, see \cite{DZ,FW}.

\begin{custom_assumption}{A-7}
    \label{assu:stable}
    The open domain $\cDc$ is stable by the vector field $-\nabla W_a$. Moreover, for any $z\in\partial \cDc$, $\ds\lim_{t\to+\infty}\psi_t(z) = a$.
\end{custom_assumption}

\begin{rem}\label{rem:D^e_D^c}
    Note that by continuity argument we can expand domain $\cDc$ such that Assumptions \ref{assu:firsttraj} and \ref{assu:stable} still hold in the enlargement. Namely, for any $\kappa > 0$ small enough there exists an open connected bounded set $\cDc_\kappa^\mathsf{e} \subseteq \{x \in \R^d: \inf_{z \in \cDc}| z - x| < \kappa \}$ such that Assumptions \ref{assu:firsttraj} and \ref{assu:stable} are satisfied for $\cDc^\mathsf{e}_\kappa$. Obviously, the same holds for constrictions: for any $\kappa > 0$ small enough there exists an open set $\cDc^\mathsf{c}_\kappa \subseteq \{x \in \cDc: \inf_{z \in \partial \cDc} | z - x | > \kappa \}$ satisfying Assumptions \ref{assu:firsttraj} and \ref{assu:stable}.

    We can also define their exit-costs as $\ds H^\mathsf{e}_\kappa := \inf_{z \in \partial \cDc^\mathsf{e}_\kappa} \{W_a(z) - W_a(a)\}$ and $\ds H^\mathsf{c}_\kappa := \inf_{z \in \partial \cDc^\mathsf{c}_\kappa} \{W_a(z) - W_a(a)\}$ respectively.
\end{rem}

%
%
%
%
%
%
%

\subsection{Existence of the process}\label{ss:Existence}

The problem of existence and uniqueness of the SDE \eqref{eq:main_SDE} was studied in \cite{HIP2}. Mutatis mutandis from {\cite[Theorem 2.13]{HIP2}}, we get the following proposition:

\begin{prop}\label{prop:existence}
    Let $r$ be the positive constant introduced in $(V-3)$. For any $\sigma \geq 0$, under Assumptions~\ref{assu:pot:V} and \ref{assu:pot:F}, the SDE~\eqref{eq:main_SDE} has a unique strong solution that we denote by $(X_t^\sigma, t \geq 0)$. Moreover, there exists a constant $M > 0$, such that 
    \begin{equation}
    \label{bishop}
        \sup_{0 \leq \sigma \leq 1}\sup_{t \geq 0} \E \big[|X_t^\sigma|^{8r^2} \big]\leq M\,.
    \end{equation}
\end{prop}

Note that the assumptions used in {\cite[Theorem 2.13]{HIP2}} are slightly different from ours, particularly for the interaction term. Assumption $(F - 2)$ of \ref{assu:pot:F} allows $\phi$ to be negative and thus to exhibit repulsive behaviour, while in \cite{HIP2} $\phi$ is set to be a positive increasing function. To neutralise possible problems that this relaxation could pose, we introduce assumption $(F-5)$. The fact that $\theta_1 > \theta_2$ guarantees that, regardless of $\mu^\sigma$, the drift term of our process is always attractive outside of a compact set. Namely, for any $\mu \in \mathcal{P}(\R^d)$ and for any $x \in \R^d$ such that $|x| > R$, we have $\nabla^2 V(x) + \nabla^2 F*\mu(x) \succeq (\theta_1 - \theta_2) {\rm Id}$ and thus 
\begin{equation*}
    \langle x; - \nabla V(x) - \nabla F*\mu(x) \rangle \leq - (\theta_1 - \theta_2)|x|^2.
\end{equation*}
This guarantees non-explosiveness of the process in finite time. After this observation, the proof in \cite{HIP2} can be easily adapted for the case of Assumptions~\ref{assu:pot:V} and \ref{assu:pot:F}.





\subsection{Large deviations principle}\label{ss:LDP}

The large deviations principle (LDP) for the process \eqref{eq:main_SDE} was also proved in \cite{HIP2}. Unlike in the case of Proposition~\ref{prop:existence}, the adaptation of these results for our assumptions on the interaction term is immediate. The authors proved the following result:

\begin{prop}[{\cite[Theorem 3.4]{HIP2}}]
\label{cor:ldp:init}
Let $\gamma$ be the unique solution of the ODE
\begin{equation*}
    \frac{\dd}{\dd{t}}{\gamma_t} = - \nabla V(\gamma_t), \quad \gamma_0 = x_\mathrm{init}.
\end{equation*}

Then for any $T > 0$, the probability measures induced by the processes $(X_t^\sigma , 0\leq t\leq T)_{\sigma > 0}$ on $\mathcal{C}([0, T])$ satisfy the LDP with convergence rate $\frac{\sigma^2}{2}$ with the following good rate function:
\begin{equation}
    \label{martika}
    I_T(\varphi):=\frac{1}{4}\int_0^T|\poin{\varphi}_t + \nabla V(\varphi_t) + \nabla F(\varphi_t - \gamma_t)|^2 \dd{t},
\end{equation}
for any $\varphi \in \cHc_1$, the set of absolutely continuous functions from $[0;T]$ to $\R^d$ such that $\varphi(0)=x_{\mathrm{init}}$. Otherwise, $I_T(\varphi):=+\infty$.
\end{prop}

If we denote by $(\nu^\sigma)_{\sigma > 0}$ the family of probability measures induced on $\mathcal{C}([0, T])$ by $(X_t^\sigma, 0\leq t\leq T)$ for respective $\sigma > 0$, then the proposition above takes the following form. For any measurable subset $\Gamma \subset \mathcal{C}([0, T])$, we have:
\begin{equation*}
        -\inf_{f \in \accentset{\circ}{\Gamma}} I(f) \leq \liminf_{\sigma \xrightarrow{} 0} \frac{\sigma^2}{2} \log \nu^\sigma(\Gamma) \leq \limsup_{\sigma \xrightarrow{} 0} \frac{\sigma^2}{2} \log \nu^\sigma(\Gamma) \leq -\inf_{f \in \overline{\Gamma}} I(f).
\end{equation*}

Note that most authors use the convergence rate $\sigma^2$ and, consequently, the term in front of the integral in~\eqref{martika} is $\frac{1}{2}$ instead of $\frac{1}{4}$. However, we choose to take as convergence rate the coefficient in front of the Laplacian in the associated partial differential equation.

Note also that in this proposition $\gamma$ represents the deterministic limit of the system~\eqref{eq:main_SDE}. When $\sigma$ is small, we expect our process to stay close to $\gamma$ for fixed time intervals. Thus it does not come as a surprise that it is $\delta_\gamma$ that replaces $\mu^\sigma$ in the rate function.

\section{Main results}\label{s:Main_results}

In this paragraph, we list the main results of the paper.

\subsection{Exit-time}\label{ss:exit-time}

We now give the main results concerning the exit-time, for the case when $\cDc$ is a bounded domain.

\begin{thm}\label{thm:main_exit_time}
    Let $H$ be the exit-cost introduced in Assumption~\ref{assu:stable}. Under Assumptions~\ref{assu:pot:V}--\ref{assu:stable}, the following two results hold
    \begin{enumerate}
        \item \textit{Kramers' law}: for any $\delta>0$, the following limit holds:
        \begin{equation}\label{eq:thm}
            \lim_{\sigma\to 0}\Prob\left[ \exp\Big\{\frac{2}{\sigma^2}(H-\delta)\Big\}\le \tau^\sigma_{\cDc} \le  \exp\Big\{\frac{2}{\sigma^2}(H+\delta)\Big\} \right]=1\,        
        \end{equation}
        \item \textit{Exit-location}: for any closed set $N \subset \partial \cDc$ such that $\inf_{z \in N} W_a(z) > H$ the following limit holds:
        \begin{equation}\label{eq:thm_exit_location}
            \lim_{\sigma \to 0} \Prob \big(X_{\tau_{\cDc}^\sigma}^\sigma \in N \big) = 0.
        \end{equation}
    \end{enumerate}
\end{thm}

Proof of Theorem~\ref{thm:main_exit_time} is provided in Section~\ref{s:main_res_proof}.

\subsection{Control of the law}\label{ss:Control_of_law}

We now present a result on the control of the law in the case where $\cDc$ is bounded. The following theorem rigorously states that, starting from some uniformly bounded in $\sigma$ time, the law of the process $\mu^\sigma$ stays close to $\delta_a$ long enough to obtain the result of Theorem \ref{thm:main_exit_time}.
\begin{thm}\label{thm:control_of_law}
    Under Assumptions~\ref{assu:pot:V}--\ref{assu:stable}, for any $\kappa > 0$ small enough there exist $\overline{T}_{\!\mathsf{st}}(\kappa) > 0$ and $\sigma_\kappa > 0$ such that
    \begin{equation*}
        \sup_{0 < \sigma < \sigma_\kappa} \sup_{t \in \Big[\overline{T}_{\!\mathsf{st}}(\kappa); \e^{\frac{2H}{\sigma^2}} \Big]} \Wass_2(\mu_t^\sigma; \delta_a) \leq \kappa.
    \end{equation*}
\end{thm}

This theorem can be easily proven using Lemmas \ref{lem:stab_in_finite_time} and \ref{lem:S_k_control} provided in Section~\ref{s:interm_results}. It is left for the reader

\subsection{Unbounded case}\label{ss:unbounded_case}

We now present the generalisation of the results above to the case where $\cDc$ is not bounded.

\begin{cor}
\label{cesare3}
    If $\cDc$ is an open and connected subset of $\R^d$, under Assumptions~\ref{assu:pot:V}--\ref{assu:pot:VetF2} and Assumptions~\ref{assu:firsttraj}, \ref{assu:stable}, the statements of Theorem~\ref{thm:main_exit_time} hold.
\end{cor}

The control of the law also holds immediately even if $\cDc$ is unbounded.

\begin{cor}
\label{cesare4}
    If $\cDc$ is an open and connected subset of $\R^d$, under Assumptions~\ref{assu:pot:V}--\ref{assu:pot:VetF2} and Assumptions~\ref{assu:firsttraj}, \ref{assu:stable}, the statement of Theorem \ref{thm:control_of_law} holds.
\end{cor}

Proofs of Corollary~\ref{cesare3} and Corollary~\ref{cesare4} are postponed to Section~\ref{s:main_res_proof}.

\subsection{Comparison to previous results}\label{ss:comparison}

In the seminal work~\cite{HIP2}, S.~Herrmann, P.~Imkeller, and D.~Peithmann proved the existence of the self-stabilizing diffusion in the irreversible case. The assumptions they used correspond to \ref{assu:pot:V} and \ref{assu:pot:F} if confinement and interaction were gradients of some regular potentials, except for a slight difference in the interaction term (this difference was discussed in Section~\ref{ss:Existence}). In the same work, the authors show the exit-time result for SSD, but, in order to do that, they had to assume convexity of confinement and interaction. Removal of this assumption, that we present in this paper, is a big improvement of previous results. Note, that, unlike in \cite{HIP2}, we solve the exit-time problem for the reversible case (confinement and interaction are gradients of some regular functions). Nevertheless, we could treat the general situation, see Section~\ref{s:extensions} on the possible extensions of our results.

Another difference between our approach and the one presented in the paper \cite{HIP2} is that, after controlling the law of the process $X^\sigma$, we use coupling techniques to prove the exit-time, while the approach used by S.~Herrmann, P.~Imkeller, and D.~Peithmann consists in reconstructing the Freidlin-Wentzell techniques and taking advantage of the contractivity of the drift.

In \cite{EJP}, J.~Tugaut focused on the reversible case of the SSD with potentials~$V$ and~$F$ being convex. He proved a similar to ours result by using another method than in \cite{HIP2}. The approach of \cite{EJP} was to apply the Freidlin-Wentzell theory without adapting it to the McKean-Vlasov diffusions. In this work, the classical large deviations principle theory for processes is used to the associated system of particles
\begin{equation}\label{eq:system_of_particles}
    \dd{X_t^{i, N}} = \sigma \dd{B_t} - \left( \nabla V(X_t^{i, N}) + \frac{1}{N}\sum_{j = 1}^N \nabla F(X_t^{i, N} - X_t^{j, N})\right) \dd{t},
\end{equation}
after which a trajectorial uniform propagation of chaos is established. Using the propagation of chaos, the author obtained the Kramers' type law.

In \cite{ECP} J.~Tugaut employed a different method, applicable to the case where the parts of the drift term are not necessarily assumed to be gradients of a regular function, although they remain globally contractive. This method primarily revolves around controlling the law at time $t$ of $X^\sigma$, denoted as $\mu_t^\sigma$. Notably, J.~Tugaut demonstrated that this law converges to $\delta_a$ in Wasserstein distance for $t \to +\infty$. Subsequently, a synchronous coupling with a diffusion, where the drift is represented as $x\mapsto-\nabla V(x)-\nabla F\ast\delta_a(x)$ instead of $x\mapsto-\nabla V(x)-\nabla F \ast \mu^\sigma_t(x)$, is employed. Exploiting the contractivity, it is straightforward to prove that the two diffusions remain close. Consequently, the exit-time of $X^\sigma$ behaves similarly to that of the coupled diffusion.

This approach has been extended to non-convex scenarios in the reversible case, as described in \cite{JOTP2}. In this context, $V$ is not necessarily convex, although $F$ exhibits sufficient convexity to ensure convexity of the effective potential $W_a = V + F(\cdot - a)$. As a result, coupling between the two diffusions is straightforward, allowing us to infer the exit-time of $X^\sigma$ from that of the coupled diffusion.

The convexity assumption on $W_a$ has been removed in \cite{Alea}, though this result is limited to the one-dimensional case. Unfortunately, the method used there cannot be directly extended to the general-dimensional case. Thus, it becomes essential to find an alternative way to control the law.

In \cite{COSA}, J.~Tugaut demonstrated that $\mu^\sigma$ does not always converge to $\delta_a$. This limitation arises when $W_a$ fails to reach its global minimum at $a$, therefore, in order to control the law of the process (at least until exit-time) other methods should be used.

Despite all these developments, the exit-time problem for SSD with general (non-convex) coefficients was an open problem throughout all these years. We solve it in this paper by significantly improving the coupling method introduced in \cite{ECP}.

\subsection{Discussions on extension}\label{s:extensions}

In this section, we provide some possible extensions to our results.

\subsubsection{Non-identity matrix as the diffusion coefficient}

In this work, we have simplified the study by assuming that the diffusion coefficient takes the form $\sigma \text{Id}$. However, for certain algorithmic applications such as molecular dynamics, it could be beneficial to consider scenarios where the diffusion coefficient is not directly proportional to the identity matrix, as discussed for example in \cite{clubdes5}.

To make further progress, it would be a significant improvement to include the scenario where the diffusion coefficient is given by $\sigma M$, with $M$ being a non-degenerate matrix. This particular situation has been studied in, for instance, \cite{DT1, DT2, Monmarche1}.

The techniques developed in the present work can be readily adapted for this non-identity diffusion coefficient case.

However, a more challenging extension would involve considering cases where $M$ is degenerate. This would allow us to address the Langevin kinetic diffusion, where both position and velocity play crucial roles. Combining techniques we have developed with those from \cite{clubdes5}, we firmly believe that we can obtain valuable insights into the asymptotic behaviour of the first exit-time.

\subsubsection{Initial random variable}

Another possible extension is related to the initial random variable. In the current work, we establish the asymptotic behaviour of the exit-time for $X_0 := x_{\mathrm{init}} \in \R^d$. However, for studying the basins of attraction, as was done in~\cite{KRM}, it is crucial to consider scenarios where $\mu_0^\sigma := \Law(X_0^\sigma)$ is not necessarily a Dirac measure. Specifically, we may be interested in cases where $\mu_0^\sigma:=\mu_0$, with the measure $\mu_0$ being compactly supported in $\cDc$.

In this situation, we need to make a slight modification to Assumptions~\ref{assu:firsttraj}. Instead of considering $\gamma'(t)=-\nabla V(\gamma_t)$, we would need to consider the partial differential equation:

\begin{equation*}
    \frac{\partial}{\partial t}\mu_t^0={\rm div}\left(\mu_t^0(\nabla V+\nabla F\ast\mu_t^0)\right),
\end{equation*}
with $\mu_0^0=\mu_0$. This corresponds to the granular media equation with zero noise. The associated dynamical system that approximates the diffusion $X^0$ on $[0; T]$ (with $T>0$) due to the large deviations principle is thus given by:
\begin{equation*}
    \rho_t(x_{\mathrm{init}})= x_\mathrm{init} - \int_0^t\nabla V(\rho_s(x_\mathrm{init})) \dd{s} - \int_0^t \nabla F\ast\mu_s^0(\rho_s(x_\mathrm{init})) \dd{s},
\end{equation*}
for any $x_{\mathrm{init}} \in \supp(\mu_0)$. In this case, Assumptions~\ref{assu:firsttraj} would be: for any $x_{\mathrm{init}}\in \cDc\cap \supp(\mu_0)$ and for any $t \geq 0$, we have $\rho_t(x_\mathrm{init}) \in  \cDc$.

The techniques developed in the present work can be seamlessly adapted to handle this situation.

\subsubsection{Reflexion on the boundary}

In this work, the diffusion process takes place in the entire phase space $\R^d$. However, we can consider a subspace of $\R^d$ instead. This could be achieved by introducing a reflection on certain boundaries, as it was done, for example, in~\cite{Tanaka}. Such an extension would be a significant improvement compared to~\cite{articledemerde}, where the uniform convexity of both confinement and interaction potentials was assumed.

In the mentioned article, the domain $\cGc$ in which the diffusion takes place satisfies~$d\!\left(\overline{\cDc}; \partial \cGc\right) > 0$, which simplifies the study. We believe that techniques we have developed could treat this case. However, considering scenarios where $\cDc \cap \cGc^c \neq \emptyset$ is more challenging. This could require extending the large deviation techniques for processes with reflection, something that is not done yet even for linear case.

\subsubsection{More accurate estimates}

In this paper, our focus has been on establishing the Kramers' law, that is a limit in probability of $\frac{\sigma^2}{2}\log(\tau^\sigma_{\cDc})$ as $\sigma$ approaches $0$, as well as the exit-location result. However, in~\cite{HIP2}, the authors have obtained a more precise estimate, which could be of interest in our context. For example, the so-called Arrhenius law was established, i.e. the convergence of
\begin{equation*}
    \frac{\sigma^2}{2} \log \E (\tau_\cDc^\sigma) \xrightarrow[\sigma \to 0]{} H > 0.
\end{equation*}
Unfortunately, since we do not provide the control of the law of the process after the exit-time, we could not use the standard method to show the Arrhenius law in the current work.

Additionally, it is well-known, as discussed in~\cite{ET:exponential}, that the first exit-time $\tau^\sigma_\cDc$ for a linear (It\^{o}) diffusion satisfies the following limit:

\begin{equation*}
    \frac{\tau^\sigma_\cDc}{\E[\tau^\sigma_\cDc]} \xrightarrow[\sigma \to 0]{\Law} \mathcal{E}(1),
\end{equation*}
where the convergence is meant in law, and $\mathcal{E}(1)$ is the exponential law with a parameter equal to $1$. The same behaviour for self-stabilizing diffusions is not established yet even in the case where both $V$ and $F$ are convex.

In \cite{Bovier1, Bovier2}, A.~Bovier, M.~Eckhoff, V.~Gayrard, and M.~Klein studied the exit-time problem for linear reversible diffusion process using potential theory approach. Using these techniques, the authors could not only establish the Arrhenius law for multi-well potential in $\R^d$, but also prefactor of the convergence. Namely, the following equality was established:
\begin{equation*}
    \E[\tau^\sigma_\cDc] = C^*\e^{\frac{2H}{\sigma^2}}\big(1 + O(\sigma|\log(\sigma)|)\big),
\end{equation*}
where the constant $C^* > 0$ depends on the derivatives of the potential $V$ at the point of attraction $a$ as well as the saddle points surrounding the well under consideration. For the explicit form of the prefactor see \cite{Bovier1}.

Similar methods could be also used for the self-stabilizing diffusion. However, that would imply studying the associated PDE for the law of the process:
\begin{equation*}
    \frac{\partial}{\partial t}\mu_t^\sigma = \frac{\sigma^2}{2}\Delta \mu_t^\sigma + {\rm div}\left(\mu_t^\sigma(\nabla V + \nabla F\ast\mu_t^\sigma)\right),  
\end{equation*}
which is considered to be a hard problem due to its non-linearity. These questions could be the focus of future studies.

\subsubsection{Non-reversible case}

In this work, we have focused on the case where both the confinement and the interaction terms are gradients of some potentials. However, it would be valuable to consider non-reversible situations of the form:

\begin{equation*}
    X_t = X_0 + \sigma M B_t + \int_0^t a(X_s) \dd{s} + \int_0^t b\ast\mu_s^\sigma(X_s) \dd{s},
\end{equation*}

where $a$ and $b$ are general vector fields on $\R^d$. It is worth noting that in previous works such as~\cite{clubdes5, HIP2, ECP}, the authors have successfully addressed this problem, but in the contractive (convex confinement and interaction) case.

The techniques developed in this paper can readily be adapted to handle the non-reversible case. However, the exit-cost is not explicit in this situation, which is why we have described the reversible case here.

\subsubsection{More general McKean-Vlasov diffusions}

A broader class of nonlinear diffusion processes can be considered. For example:

\begin{equation*}
    \dd{X_t} = \sigma \dd B_t - \nabla V(X_t) \dd{t}-b(X_t,\mu_t^\sigma) \dd{t},
\end{equation*}

where the nonlinear drift $b$ takes the form

\begin{equation*}
    b(x,\mu):= \int_{\R^d} B(x,y) \mu(\dd{y}).
\end{equation*}

Here, the function $B$ is required to be regular and maps from $\R^{d} \times\R^d$ to $\R^d$. Such a generalization would have significant implications for theoretical purposes (as shown in \cite{GvalaniSchlichting}) as well as applications (see e.g. \cite{Maurelli-batteries}). We firmly believe that the techniques developed in this work can be adapted to handle a wide range of situations within this framework.

For algorithmic applications, it would be also interesting to include jumps in the process, as discussed in~\cite{Graham2, Graham1}. This could be a subject of future studies.

\subsubsection{Extension on the domain \texorpdfstring{$\cDc$}{D} and metastability}

An important yet challenging extension concerns the domain $\cDc$ itself. In this work, we have confined our study to cases where $\overline{\cDc}$ is stable under the effective potential $W_a$. However, the most interesting scenario arises when the saddle point lies on the boundary of $\cDc$.

Moreover, it would be interesting to establish some metastable properties of $X^\sigma$, that is considering $t(\sigma)$ as a function of $\sigma$ and investigating $X^\sigma_{t(\sigma)}$ in metastable confinement as it was done in \cite{FW}. Complexity of this problem in the case of SSD is that the drift itself (the effective potential) may change after the transition of the process from one metastable state to another. These questions could be the focus of future studies.

\subsubsection{System of particles}

For algorithmic applications, it is essential to consider the associated system of particles described by Equation~\eqref{eq:system_of_particles}. In this system, the measure $\mu_t^X$ is replaced by $L_t^\sigma := \frac{1}{N} \sum_{j=1}^N \delta_{X_t^j}$.

In \cite{EJP}, J.~Tugaut has obtained the exit-time of the McKean-Vlasov diffusion from the system of particles in the convex case. Consequently, it appears feasible to do the opposite and establish the exit-time of the system of particles based on the exit-time of the McKean-Vlasov diffusion. Similar techniques like a trajectorial uniform propagation of chaos (see for example \cite{BRTV, CGM, Sznitman}) can be used. However, in \cite{EJP}, convexity was essential for controlling the law, which is now also available in the general situation due to the current work.

\section{Intermediate results}\label{s:interm_results}

In this preliminary section, we will give the key results which allow us to prove the main theorems related to exit-time in Section~\ref{s:main_res_proof}. Their proofs are given in Section~\ref{s:interm_res_proof}.

\subsection{Stabilisation in finite time}

Let us define the following two deterministic times for any $\kappa > 0$ small enough:
\begin{equation*}
\begin{aligned}
    T_{\mathsf{st}}^\sigma(\kappa) &:= \inf\left\{t \geq 0\,\,:\,\,\Wass_2(\mu_t^\sigma; \delta_a) \leq \kappa\right\}, \\
    S_{\mathsf{st}}^\sigma(\kappa) &:=  \inf\left\{t \geq T_\mathsf{st}^\sigma(\kappa)\,\,:\,\,\Wass_2(\mu_t^\sigma; \delta_a) > \kappa\right\};
\end{aligned}
\end{equation*}
and we let the infima to be equal to $+\infty$ if respective sets are empty. 

First key result consists in obtaining the existence of a time $T$ such that $\Wass_2(\mu_t^\sigma; \delta_a)$ is small and such that $X_T^\sigma$ is concentrated around $a$.

\begin{lem}
\label{lem:stab_in_finite_time}
    Under Assumptions~\ref{assu:pot:V}--\ref{assu:stable}, for any $\kappa > 0$ there exist $\overline{T}_{\!\textsf{st}}(\kappa) > 0$ and $\sigma_\kappa > 0$ such that:
    \begin{equation*}
        T_{\textsf{st}}^\sigma(\kappa) \leq  \overline{T}_{\!\textsf{st}}(\kappa)\quad \text{for any } 0 < \sigma < \sigma_\kappa.
    \end{equation*} 
    Moreover,
    \begin{equation*}
        \lim_{\sigma \to 0} \Prob\! \left( \left|X_{\overline{T}_{\!\textsf{st}}(\kappa)}^\sigma - a \right| > \kappa \right) = 0.
    \end{equation*}
\end{lem}

An important implication of this lemma is that, with high probability, the exit from the domain $\cDc$ does not occur before time $\overline{T}_{\!\mathsf{st}}(\kappa)$ (see Section~\ref{s:interm_res_proof} for the proof). Consider the following corollary.  

\begin{cor}
\label{cor:stab_in_finite_time2}
    Under Assumptions~\ref{assu:pot:V}--\ref{assu:stable}, for any $\kappa>0$ the following limit holds:
\begin{equation*}
    \lim_{\sigma\to0} \Prob\!\left(\tau^\sigma_\cDc \leq \overline{T}_{\!\mathsf{st}}(\kappa) \right) = 0\,.
\end{equation*}
\end{cor}

\subsection{The coupling method}

We now introduce the diffusion $Y^\sigma := (Y_t^\sigma,\, t \geq T_{\mathsf{st}}^\sigma(\kappa))$ solution to the following linear SDE:

\begin{equation}\label{eq:def:Y}
\begin{aligned}
    Y_t^\sigma &= X_{T_{\mathsf{st}}^\sigma(\kappa)}^\sigma + \sigma (B_t - B_{T_{\mathsf{st}}^\sigma(\kappa)}) - \int_{T_{\mathsf{st}}^\sigma(\kappa)}^t \nabla V(Y_s^\sigma) \dd{s} \\
    & \quad  - \int_{T_{\mathsf{st}}^\sigma(\kappa)}^t\nabla F\left(Y_s^\sigma - a \right) \dd{s},
\end{aligned}
\end{equation}

where $(B_t, t \geq 0)$ is the same Brownian motion that drives the main equation \eqref{eq:main_SDE}. Note, that this SDE has a unique solution (see for example \cite[Theorem~10.2.2, p.~255]{Stroock}). 

Note also that $Y^\sigma$ is a linear diffusion. As a consequence, we can apply the classical Freidlin-Wentzell theory, see~\cite{DZ,FW}, for estimating the first exit-time as the diffusion coefficient tends to $0$.

Apart from the processes $(Y^\sigma)_{0 < \sigma < 1}$ that is defined by SDE \eqref{eq:def:Y}, we also define the following family of processes that constitute Itô diffusions and will help us to study stochastic properties of $Y^\sigma$. For any $y \in \R^d$ and for any $0 < \sigma < 1$ define $(Y^{y, \sigma}_t, t \geq 0)$ as the unique solution to the following SDE:
\begin{equation}\label{eq:def:Y_y_sigma}
    Y_t^{y, \sigma} = y + \sigma B_t - \int_0^t \nabla V(Y_s^{y, \sigma}) \dd{s} - \int_0^t \nabla F(Y_s^{y, \sigma} - a)\dd{s}.
\end{equation}

Following the standard notation for diffusions, we will drop the initial point $y$ for $Y^{y, \sigma}$, as well as for all random variables that are functions of $Y^{y, \sigma}$, and put it as a subscript under the probability measure. Namely, for any $y \in \R^d$ we introduce a probability measure $\TrProb_y$ that is simply a restriction of $\Prob$ to the measurable space $\big(\Omega, \sigma(Y^{y, \sigma}_t: t\geq 0)\big)$.

The following proposition is a classical result of Freidlin--Wentzell theory for the exit-time of linear diffusions of the type \eqref{eq:def:Y_y_sigma}. Consider:

\begin{prop}[\cite{DZ}, Theorem 5.7.11]\label{prop:DZ_exit_time}
    Let Assumption~\ref{assu:pot:V} be satisfied and let $G \subset \R^d$ be a domain such that Assumptions~\ref{assu:domain:bounded}--\ref{assu:stable} are satisfied for it and its exit-cost $\ds H_G := \inf_{z \in \partial G} \{W_a(z) - W_a(a)\}$. Let $K \subset G$ be a compact set. Define $\tau^{Y, \sigma}_G := \inf\{t \geq 0: Y^{y, \sigma}_t \notin G\}$. Then, for any $\delta > 0$ we have
    \begin{equation*}
        \lim_{\sigma \to 0} \sup_{y \in K} \TrProb_y\! \left(\exp{\frac{2(H_G - \delta)}{\sigma^2}} \leq \tau^{Y, \sigma}_{G} \leq \exp{\frac{2(H_G + \delta)}{\sigma^2}} \right) = 1.
    \end{equation*}
\end{prop}

Obviously, this theorem also holds when $G$ is the domain $\cDc^\mathsf{e}_\kappa$ defined as in Remark~\ref{rem:D^e_D^c} and $H_G = H^\mathsf{e}_\kappa := \inf_{x \in \partial \cDc^\mathsf{e}_\kappa} \{W_a(x) - W_a(a)\}$ respectively.




Let us now describe how both diffusion processes $X$ (the targeted diffusion) and $Y$ (the auxiliary one) are coupled. We are especially interested in describing the distance between them.

\begin{prop}\label{prop:|X_t-Y_t|}
    Under Assumptions~\ref{assu:pot:V}--\ref{assu:stable} there exists $\eta > 0$ such that for any $\kappa > 0$ small enough, we have
    \begin{equation*}
        \lim_{\sigma \to 0} \Prob( \sup  \big|X_t^\sigma - Y_t^\sigma \big|  > \kappa) = 0,
    \end{equation*}
    where supremum is taken over $t \in \left[T_{\mathsf{st}}^\sigma(\kappa); S_\mathsf{st}^\sigma(\kappa) \wedge \exp{\frac{2(H + \eta)}{\sigma^2}} \right]$.
\end{prop}

As it is shown below (Corollary~\ref{cor:|X_t-Y_t|}), this result can be improved by removing the time $S_\mathsf{st}^\sigma(\kappa)$, since, as it turns out, the destabilization of the law of the process can not happen before its exit-time from the domain $\cDc$. 

The following lemma is an important result stating that, at each point of time, the diffusion $Y^\sigma$ is close to $a$ with high probability.

\begin{lem}\label{lem:Y_t_notin_B_rho}
     Let $\rho$ be a positive constant introduced in Definition~\ref{def:Loc_convexity}. Under Assumptions~\ref{assu:pot:V}--\ref{assu:stable} there exists $\eta > 0$ small enough such that for any $\kappa > 0$ small enough:
     \begin{equation*}
         \sup \Prob \! \left(Y_t^\sigma \notin B_{\rho/2}(a) \right) = o_\sigma(1),
     \end{equation*}
    where supremum is taken over $t \in \left[T_{\mathsf{st}}^\sigma(\kappa); \exp{\frac{2(H + \eta)}{\sigma^2}} \right]$.     
\end{lem}

Note, that the position of supremum in Lemma~\ref{lem:Y_t_notin_B_rho} is important. Indeed, according to the Freidlin-Wentzell theory for It\^{o} diffusions, the exit-time of $Y^\sigma$ from $B_{\rho/2}(a)$ is, with high probability, of order $\exp{2 H_{\rho/2}/\sigma^2}$, where $H_{\rho/2} := \inf_{z \in \partial B_{\rho/2}(a)} \{V(z) + F(z - a) - V(a)\}$, which means, among other things, that we can not expect $\Prob \left(\sup |Y_t^\sigma - a| > \frac{\rho}{2} \right)$ to be equal to $o_\sigma(1)$. Instead, what Lemma~\ref{lem:Y_t_notin_B_rho} states is that for all $t$ before the exit of $Y^\sigma$ from a small enlargement $\cDc^\mathsf{e}_\kappa$, the probability that $Y^\sigma$ is not close to $a$ tends to 0. We come back to this description in Section \ref{s:proof_control_of_Y_sigma}.

\subsection{Control of the law}
In this section we present a result regarding the control of the law of the process after the stabilisation time. Consider the following lemma.



\begin{lem}\label{lem:S_k_control}
    Under Assumptions \ref{assu:pot:V}--\ref{assu:stable} there exists $\eta > 0$ such that for any $\kappa > 0$ small enough there exists $\sigma_\kappa$ such that for any $0 < \sigma < \sigma_\kappa$ we have
    \begin{equation*}
        S_{\mathsf{st}}^\sigma(\kappa) > \exp{\frac{2(H + \eta)}{\sigma^2}}.
    \end{equation*}
\end{lem}


This lemma together with Proposition~\ref{prop:|X_t-Y_t|} immediately gives us the following corollary:

\begin{cor}\label{cor:|X_t-Y_t|}
    Under Assumptions~\ref{assu:pot:V}--\ref{assu:stable} there exists $\eta > 0$ such that for any $\kappa > 0$ small enough, we have
    \begin{equation*}
        \lim_{\sigma \to 0} \Prob( \sup  \big|X_t^\sigma - Y_t^\sigma \big|  > \kappa) = 0,
    \end{equation*}
    where supremum is taken over $t \in \left[T_{\mathsf{st}}^\sigma(\kappa); \exp{\frac{2(H + \eta)}{\sigma^2}} \right]$.    
\end{cor}

\section{Proofs of the main results}\label{s:main_res_proof}

Here, we give the proofs of the main results.

\subsection{Exit-time and exit-location}

\noindent{}{\bf Step 1.} To prove the lower bound of Kramers' law, consider the following inequality. For any $\delta > 0$ and for fixed $\kappa > 0$ small enough we have
\begin{equation}\label{eq:aux:tau^Y<exp}
\begin{aligned}
    \Prob &\left(\tau_{\cDc}^\sigma < \exp{\frac{2(H - \delta)}{\sigma^2}} \right) \leq \Prob(\tau_{\cDc}^\sigma < T_{\mathsf{st}}^\sigma(\kappa)) \\
    & \quad +  \Prob \Big(\tau_{\cDc}^\sigma < \exp\Big\{\frac{2(H - \delta)}{\sigma^2}\Big\}, \sup_{t \in [T_{\mathsf{st}}^\sigma(\kappa); \e^{\frac{2 H}{\sigma^2}}]}|X_t^\sigma - Y_t^\sigma| \leq \kappa \Big) \\
    & \quad +  \Prob \Big( \sup_{t \in [T_{\mathsf{st}}^\sigma(\kappa); \e^{\frac{2 H}{\sigma^2}}]}|X_t^\sigma - Y_t^\sigma| > \kappa \Big).
\end{aligned}
\end{equation}
By the construction of the domain $\cDc_\kappa^\mathsf{c}$ (see Remark \ref{rem:D^e_D^c}), $d(\cDc_\kappa^\mathsf{c}, \partial\cDc) \geq \kappa$. Let us define $\delta_\kappa := H - H_\kappa^\mathsf{c}$. Note that $H_\kappa^\mathsf{c} \xrightarrow[\kappa \to 0]{} H$ due to the continuiuty of the effective potential $W_a$. Therefore, we can choose $\kappa$ to be small enough such that $\delta_\kappa < \delta$. Then the following inequality holds:
\begin{equation*}
\begin{aligned}
    \Prob &\left(\tau_{\cDc}^\sigma < \exp{\frac{2(H - \delta)}{\sigma^2}},\quad \sup_{}|X_t^\sigma - Y_t^\sigma| \leq \kappa \right) \\
    &\leq \Prob\Big(\tau^{Y, \sigma}_{\cDc_\kappa^\mathsf{c}} >  \exp\Big\{\frac{2(H - \delta)}{\sigma^2}\Big\} = \exp\Big\{\frac{2(H_\kappa^\mathsf{c} + \delta_\kappa- \delta)}{\sigma^2}\Big\} \Big)\\
    &\leq \Prob(|X_{T_{\mathsf{st}}^\sigma(\kappa)}^\sigma - a| > \kappa) + \sup_{y \in B_\kappa(a)}  \!\!\TrProb_y \left(\tau^{Y, \sigma}_{\cDc_\kappa^\mathsf{c}} >  \exp{\frac{2(H_\kappa^\mathsf{c} + \delta_\kappa- \delta)}{\sigma^2}} \right) \xrightarrow[\sigma \to 0]{} 0,    
\end{aligned}
\end{equation*}
where the convergence to $0$ is due to Lemma~\ref{lem:stab_in_finite_time} and Proposition~\ref{prop:DZ_exit_time}, since $\delta_\kappa - \delta < 0$. 

The other probabilities in \eqref{eq:aux:tau^Y<exp} converge to $0$ by Corollaries~\ref{cor:stab_in_finite_time2} and \ref{cor:|X_t-Y_t|}.

\noindent{}{\bf Step 2.} To prove the upper bound of Kramers' law, consider the set $\cDc^\mathsf{e}_\kappa$ (see Remark \ref{rem:D^e_D^c}): enlargement of $\cDc$ for small enough $\kappa > 0$. Let $\eta > 0$ be the positive constant defined in Corollary~\ref{cor:|X_t-Y_t|}. Without loss of generality, let us fix positive $\delta < \eta$. Consider the following inequalities.
\begin{equation}\label{eq:aux:tau^Y>exp}
    \begin{aligned}
        \Prob&\left(\tau_\cDc^\sigma > \exp{\frac{2(H + \delta)}{\sigma^2}} \right) \leq \Prob(\tau_\cDc^\sigma < T_{\mathsf{st}}^\sigma(\kappa)) \\
        & \quad + \Prob\Big( \tau_\cDc^\sigma > \exp\Big\{\frac{2(H + \delta)}{\sigma^2}\Big\}, \sup_{t \in [T_{\mathsf{st}}^\sigma(\kappa); \e^{\frac{2(H + \delta)}{\sigma^2}}]} |X_t^\sigma - Y_t^\sigma| \leq \kappa \Big) \\
        & \quad + \Prob \Big( \sup_{t \in [T_{\mathsf{st}}^\sigma(\kappa); \e^{\frac{2(H + \delta)}{\sigma^2}}]} |X_t^\sigma - Y_t^\sigma| > \kappa \Big).
    \end{aligned}
\end{equation}

If $\tau_\cDc^\sigma > \exp{\frac{2(H + \delta)}{\sigma^2}}$ and $\sup \Big\{|X_t^\sigma - Y_t^\sigma|: t \in [T_{\mathsf{st}}^\sigma(\kappa); \e^{\frac{2(H + \delta)}{\sigma^2}}] \Big\} \leq \kappa$, then at the point of time $\e^{\frac{2(H + \delta)}{\sigma^2}}$ the process $Y^\sigma$ is still inside $\cDc_\kappa^\mathsf{e}$. Define $\delta_\kappa := H^\mathsf{e}_\kappa - H$, decrease $\kappa$ if necessary such that $\delta_\kappa < \delta$, and consider
\begin{equation*}
\begin{aligned}
    \Prob &\left(  \tau_\cDc^\sigma > \exp{\frac{2(H + \delta)}{\sigma^2}}, \sup_{t \in [T_{\mathsf{st}}^\sigma(\kappa); \e^{\frac{2(H + \delta)}{\sigma^2}}]} |X_t^\sigma - Y_t^\sigma| \leq \kappa \right)\\   
    & \leq \Prob \Big(\tau_{\cDc_\kappa^\mathsf{e}}^{Y, \sigma} > \exp\Big\{\frac{2(H + \delta)}{\sigma^2}\Big\} = \exp\Big\{\frac{2(H_\kappa - \delta_\kappa + \delta)}{\sigma^2}\Big\} \Big)\\
    &\leq \Prob(|X_{T_{\mathsf{st}}^\sigma(\kappa)}^\sigma - a| > \kappa) + \sup_{y \in B_\kappa(a)}  \!\!\TrProb_y \left(\tau_{\cDc_\kappa^\mathsf{e}}^{Y, \sigma} > \exp{\frac{2(H_\kappa - \delta_\kappa + \delta)}{\sigma^2}} \right) \xrightarrow[\sigma \to 0]{} 0,
\end{aligned}
\end{equation*}
where the convergence to $0$ holds due to Lemma~\ref{lem:stab_in_finite_time} and Proposition~\ref{prop:DZ_exit_time}. We finalise the proof of Kramers' type law by observing that, as in Step 1, all the others probabilities  in \eqref{eq:aux:tau^Y>exp} also tend to $0$ by Corollaries~\ref{cor:stab_in_finite_time2} and \ref{cor:|X_t-Y_t|}. That proves Kramers' type law.

\noindent{}{\bf Step 3.} Let us now show the exit-location result. Fix a set $N \subset \partial \cDc$ such that $\ds \inf_{z \in N} \{W_a(z) - W_a(a)\} > H$. Let us choose $\xi > 0$ to be small enough such that $\ds \xi < \big(\inf_{z \in N} \{W_a(z) - W_a(a)\} - H \big)/2$. Let us define the sublevel set $L^{-}_{H + \xi} := \{x \in \R^d: W_a(x) - W_a(a) \leq H + \xi\}$ (without loss of generality by $L^{-}_{H + \xi}$ we will denote the unique connected component of the sublevel set that contains $a$). By geometric properties of the effective potential (regularity and convergence at infinity for big $|x|$), $L^{-}_{H + \xi}$ satisfies the Assumptions \ref{assu:domain:bounded}--\ref{assu:stable}. Thus, after the initial convergence of $X^\sigma$ to $a$ and its law $\mu_t^\sigma$ to $\delta_a$, the Kramers' type law holds for the exit-time $\tau_{L^{-}_{H + \xi}}^\sigma$, that is, for any $\delta > 0$,
\begin{equation}\label{eq:aux:tau_L-H}
    \lim_{\sigma \to 0}\Prob \left(\e^{\frac{2(H + \xi - \delta)}{\sigma^2}} \leq \tau_{L^{-}_{H + \xi}}^\sigma \leq \e^{\frac{2(H + \xi + \delta)}{\sigma^2}}\right) = 0,
\end{equation}
including for $\delta = \xi/2$. We could easily show geometrically that exiting $\cDc$ in the set $N$ means crossing the boundary $L_{H + \xi} := \partial L^-_{H + \xi}$ before leaving the domain $\cDc$. Therefore, we get the following inequality:
\begin{equation*}
    \Prob (X_{\tau_\cDc^\sigma}^\sigma \in N) \leq \Prob (\tau_\cDc^\sigma \leq T_{\text{st}}^\sigma(\kappa)) + \Prob (\tau_{L^-_{H + \xi}}^\sigma \leq \tau_\cDc^\sigma).
\end{equation*}

The first probability converges to 0 by Corollary~\ref{cor:stab_in_finite_time2}. Let us look at the second probability:
\begin{equation*}
    \Prob (\tau_{L^-_{H+\xi}}^\sigma \leq \tau_\cDc^\sigma) \leq \Prob \left(\tau_\cDc^\sigma \geq \e^{\frac{2(H + \xi/2)}{\sigma^2}} \right) + \Prob \left(\tau_{L^-_{H + \xi}}^\sigma \leq \tau_\cDc^\sigma < \e^{\frac{2(H + \xi/2)}{\sigma^2}}\right) \xrightarrow[\sigma \to 0]{} 0,
\end{equation*}
where th first probability tends to 0 by the Kramers' type law (Step 2) and the second probability tends to 0 by \eqref{eq:aux:tau_L-H} if we take $\delta = \xi/2$.

\subsection{Proof of Corollaries~\ref{cesare3} and \ref{cesare4}} 
We consider an unbounded domain $\cDc$ with finite exit-cost $H>0$. Then, set $L_{H + \xi}^- := \left\{x\in\R^d\,\,:\,\,W_a(x)-W_a(a)\leq H+\xi\right\}$. Let us assume without loss of generality that $x_{\text{init}} \in L_{H + \xi}^-$ (otherwise, the uniform in $\sigma$ convergence in finite time inside $L_{H + \xi}^-$ can be easily proven using LDP, similarly to Lemma~\ref{lem:stab_in_finite_time}).

Let us define $\cDc':=\cDc \bigcap L_{H + \xi}^-.$ Immediately, $\cDc'$ is bounded. Indeed, since $W_a(x)$ tends to infinity as $|x|$ goes to infinity, the level set $L_{H + \xi}^-$ is compact. The domain $\cDc'$ is also stable by $-\nabla W_a$, since both the domains $\cDc$ and $L_{H + \xi}^-$ are stable by definition. Thus, the domain $\cDc'$ satisfies all the assumptions of Theorem~\ref{thm:main_exit_time} with the height of $W_a$ inside $\cDc'$ being equal to $H$. Therefore, for any $\xi > 0$ we have:

\begin{equation*}
    \lim_{\sigma\to0} \Prob\left(\e^{\frac{2}{\sigma^2}(H-\xi)} \leq \tau'(\sigma) \leq \e^{\frac{2}{\sigma^2}(H+\xi)} \right) = 1\,,
\end{equation*}

were, $\tau'(\sigma)$ is the first exit-time of $X^\sigma$ from $\cDc'$. Indeed, the exit-cost is $H$.

Note that, by construction of the domain $\cDc'$, and by continuity of $W_a$, for any $\xi > 0$ we have $$\inf \big\{W_a(z) - W_a(a): z \in \Cl (\partial\cDc' \setminus \partial\cDc) \} > H,$$ 
where $\Cl$ stands for closure. It means that the exit-location result of the main Theorem~\ref{thm:main_exit_time} holds for $N = \Cl(\partial \cDc' \setminus \partial \cDc)$, namely
\begin{equation*}
    \lim_{\sigma \to 0} \Prob\left(X^\sigma_{\tau'(\sigma)} \in \Cl(\partial \cDc' \setminus \partial \cDc)\right) = 0.   
\end{equation*}
That essentially means that
\begin{equation*}
    \lim_{\sigma \to 0}\Prob (\tau'(\sigma) = \tau^\sigma_\cDc) = 1,
\end{equation*}
which proves Corollary~\ref{cesare3}.

The second corollary can be proved the same way by choosing $\xi > 0$ to be small enough such that the set under consideration $N \subset \cDc$ lies entirely beyond the level set $L^-_{H + \xi}$.

\section{Proofs of the intermediate results} \label{s:interm_res_proof}

\subsection{Stabilisation in finite time: Proof of Lemma \ref{lem:stab_in_finite_time} and Corollary \ref{cor:stab_in_finite_time2}}

The proof is based on LDP ideas and the fact that, for small $\sigma$, the process $X^\sigma$ is attracted towards $a$. Fix some $\kappa > 0$. By Assumption~\ref{assu:firsttraj}, the path of the deterministic solution to the following equation
\begin{equation}\label{eq:aux:gamma=-nabla_V(gamma)}
    \frac{\dd}{\dd{t}}\gamma_t = - \nabla V(\gamma_t), \quad \text{with } \gamma_0 = x_{\text{init}}, 
\end{equation}
is contained in $\cDc$, i.e. $\{\gamma_t, t\geq 0\} \subset \cDc$, and tends to $a$. Let us decrease $\kappa > 0$ to be small enough such that the distance between the set $(\gamma_t, t\geq 0)$ and $\partial \cDc$ is strictly greater than $\kappa/3$. Let us define $\overline{T}_{\mathsf{st}}(\kappa)$ as the first time when $\gamma_t \in B_{\kappa/3}(a)$. The following inclusion of events takes place:
\begin{equation*}
    \Prob \left(\Big| X_{\overline{T}_{\mathsf{st}}(\kappa)}^\sigma - a \Big| > \frac{2\kappa}{3} \right) \leq \Prob \left(\Big| X_{\overline{T}_{\mathsf{st}}(\kappa)}^\sigma - \gamma_{\overline{T}_{\mathsf{st}}(\kappa)} \Big| > \frac{\kappa}{3} \right) \leq \Prob (X^\sigma \in \Phi),
\end{equation*}
where $\Phi := \big\{\varphi \in C\left( \left[ 0; \overline{T}_{\mathsf{st}}(\kappa) \right] \right): \|\varphi - \gamma\|_{\infty} \geq \kappa/3 \big\}$. By Proposition~ \ref{cor:ldp:init}, 
\begin{equation}\label{eq:aux:Prob(|X_T-a|)}
    \limsup_{\sigma \to 0} {\frac{\sigma^2}{2} \log \Prob(X^\sigma \in \Phi)} \leq - \inf_{\varphi \in \Phi} I_{\overline{T}_{\mathsf{st}}(\kappa)}(\varphi). 
\end{equation}
Note that, by definition of the rate function $I_T$, and by uniqueness of solution to equation \eqref{eq:aux:gamma=-nabla_V(gamma)}, function $\gamma$ is its only minimizer such that $I_{\overline{T}_{\mathsf{st}}(\kappa)}(\gamma) = 0$. Since $I_T$ is a good rate function, its infima are achieved over closed sets. Note that $\gamma \notin \Phi$, thus $A:= I_{\overline{T}_{\mathsf{st}}(\kappa)}(\varphi) > 0$. That proves the second result of the Lemma~\ref{lem:stab_in_finite_time}, since it guarantees that there exists $\sigma_\kappa > 0$ small enough such that for any $0 < \sigma < \sigma_\kappa$: 
\begin{equation}\label{eq:aux:|X_t-a|>2k/3}
    \Prob \left(\Big| X_{\overline{T}_{\mathsf{st}}(\kappa)}^\sigma - a \Big| > \frac{2\kappa}{3} \right) \leq \e^{-\frac{2 A}{\sigma^2}}.
\end{equation}

For the first statement, consider the following equality:
\begin{equation*}
\begin{aligned}
    \Wass_2^2 \left(\mu_{\overline{T}_{\mathsf{st}}(\kappa)}^\sigma; \delta_a \right)  &= \E \left| X^\sigma_{\overline{T}_{\mathsf{st}}(\kappa)} - a \right|^2 = \E\left[ \left| X^\sigma_{\overline{T}_{\mathsf{st}}(\kappa)} - a \right|^2 \1_{\{X^\sigma_{\overline{T}_{\mathsf{st}}(\kappa)} \in B_{\frac{2\kappa}{3}}(a)\}} \right] \\
    & \quad + \E\left[ \left| X^\sigma_{\overline{T}_{\mathsf{st}}(\kappa)} - a \right|^2 \1_{\{X^\sigma_{\overline{T}_{\mathsf{st}}(\kappa)} \notin B_{\frac{2\kappa}{3}}(a)\}} \right].
\end{aligned}
\end{equation*}
Therefore, by Cauchy–Schwarz inequality, we can bound the difference between the two measures by:
\begin{equation*}
    \Wass_2^2 \left(\mu_{\overline{T}_{\mathsf{st}}(\kappa)}^\sigma; \delta_a \right) \leq \frac{4\kappa^2}{9} + \sqrt{\E \left| X^\sigma_{\overline{T}_{\mathsf{st}}(\kappa)} - a \right|^4} \sqrt{ \Prob \left(\Big| X_{\overline{T}_{\mathsf{st}}(\kappa)}^\sigma - a \Big| > \frac{2\kappa}{3} \right)}.
\end{equation*}

By Proposition \ref{prop:existence}, there exists $ M > 0 $ such that $\sup_{0 < \sigma < 1} \sup_{t \geq 0} \E|X^\sigma_t - a|^2 < M^2$. This estimate along with equation \eqref{eq:aux:|X_t-a|>2k/3} gives us:
\begin{equation*}
     \Wass_2^2 \left(\mu_{\overline{T}_{\mathsf{st}}(\kappa)}^\sigma; \delta_a \right) \leq \frac{4\kappa^2}{9} + M \e^{-A/\sigma^2}.
\end{equation*}
That expression can be bounded by $\kappa^2$ if we choose $\sigma_\kappa > 0$ to be small enough, which proves Lemma \ref{lem:stab_in_finite_time}.

Corollary \ref{cor:stab_in_finite_time2} can be also easily proven by choosing $\kappa$ such that $$\inf_{t \geq 0} \inf_{z \in \partial \cDc}|\gamma_t - z| > \frac{\kappa}{3}.$$ In this case, the following estimate holds:
\begin{equation*}
    \Prob \left(\tau^\sigma_\cDc \leq \overline{T}_{\mathsf{st}}(\kappa) \right) \leq \Prob(X \notin \Phi) \leq \e^{-\frac{2A}{\sigma^2}} \xrightarrow[\sigma \to 0]{} 0.
\end{equation*}

\subsection{The coupling estimate: Proof of Proposition \ref{prop:|X_t-Y_t|}}

In this section we prove Proposition \ref{prop:|X_t-Y_t|}. The idea of the proof is based on the fact that, since the processes $X^\sigma$ and $Y^\sigma$ are coupled by the same Brownian motion and by the properties of convex sets, whenever both $X^\sigma$ and $Y^\sigma$ belong to the set $B_\rho(a)$ (Definition \ref{def:Loc_convexity}), the distance between them decreases a.s. (we show this in Lemma \ref{lem:|X-Y|_control_inside}). At the same time, whenever the two processes belong to the region $\cDc \setminus B_\rho(a)$, their maximum scatter can be controlled in terms of the time spent inside $\cDc \setminus B_\rho(a)$ (Lemma \ref{lem:|X-Y|_control_outside} below). The proof is finished by observing that, before exiting $\cDc$, the processes $X^\sigma$ and $Y^\sigma$ spend inside $B_\rho(a)$ long enough time comparing to the total time spent inside $\cDc \setminus B_\rho(a)$, that the attracting effect surpasses the scattering one.     

Before proving the proposition rigorously, let us present the following notions. Let us decrease without loss of generality $\kappa > 0$ to be smaller than $\rho/4$. Let us also fix some enlargement of the domain $\cDc$ of some radius $R > 0$: $\cDc_{R}^\textsf{e}$ (see Remark \ref{rem:D^e_D^c} for the definition). Decrease $\kappa$, if necessary, so that $\kappa < R/2$. Consider the following sequence of stopping times:
\begin{equation}\label{eq:def_of_tau_theta}
    \begin{aligned}
        \theta_1 &:= \inf\{t \geq T_{\mathsf{st}}^\sigma(\kappa) :\; Y_t^\sigma \notin B_{\rho/2}(a) \},\\
        \tau_{m} &:= \inf\{t\geq \theta_{m} : Y_t^\sigma \in B_{\rho/4}(a) \cup \partial \cDc_R^\mathsf{e} \}, \\
        \theta_{m + 1} &:= \inf\{t \geq \tau_m: Y_t^\sigma \notin B_{\rho/2}(a) \}.
    \end{aligned}
\end{equation}

We also define the following stopping times that will allow us to study the behaviour of $\theta_i$, $\tau_i$ for different $i$ using the strong Markov property of diffusion $Y^\sigma$. For any $y \in \R^d$ consider:
\begin{equation}\label{eq:def_of_tau_0_theta_0}
    \begin{aligned}
        \theta_0 &:= \inf\{t \geq 0 :\; Y_t^{y, \sigma} \notin B_{\rho/2}(a) \},\\
        \tau_0 &:= \inf\{t \geq 0 : Y_t^{y, \sigma} \in B_{\rho/4}(a) \cup \partial \cDc_R^\mathsf{e} \}.
    \end{aligned}
\end{equation}

Consider the following 

\begin{lem}\label{lem:|X-Y|_control_inside}
    Define for some $K > 0$ the following family of mappings $\varphi_T: x \mapsto x\e^{- K T} + o_\kappa(1)$ for any $T > 0$, where $o_\kappa(1) \xrightarrow[\kappa \to 0]{}0$. Then, under Assumptions \ref{assu:pot:V}--\ref{assu:stable}, there exists a constant $K > 0$ such that for any $\alpha < \rho/4$, for any $m \geq 1$, and for any $\kappa > 0$ small enough:
    \begin{equation*}
        \Prob \left(\sup_{t \in [\tau_m; \theta_{m + 1}]} |X_t^\sigma - Y_t^\sigma| > \varphi_{\theta_{m + 1} - \tau_m}(\alpha), A \right) = 0,
    \end{equation*}
    where $A := \{\theta_{m + 1} \leq S_{\mathsf{st}}^\sigma(\kappa), \sup_{t \leq \tau_m} |X_t^\sigma - Y_t^\sigma| \leq \alpha \}$
\end{lem}

\begin{proof}

    Let us define random time $\mathcal{T}:= \inf\{t \geq \tau_{m}: X_t^\sigma \notin B_{\rho}(a)\}$ -- first time when $X^\sigma$ leaves the convexity area $B_\rho(a)$. Obviously, for almost every $\omega \in A$, we have $\mathcal{T} > 0$.

    \noindent\noindent{}{\bf Step 1.} Let us define $\xi(t) := |X^\sigma_t - Y^\sigma_t|^2$. The way functions $X^\sigma$ and $Y^\sigma$ are coupled provides us with the fact that $\xi$ is differentiable in the usual sense. Its derivative is equal to:
    \begin{equation*}
    \begin{aligned}
        \xi^\prime (t) &= - 2 \langle X^\sigma_t - Y^\sigma_t; \, \nabla W_a(X^\sigma_t) - \nabla W_a(Y^\sigma_t) \rangle\\
                    & \quad - 2\langle X^\sigma_t - Y^\sigma_t;\, \nabla F \ast \mu_t^\sigma(X^\sigma_t) -  \nabla F \ast \delta_a (X^\sigma_t) \rangle.         
    \end{aligned}
    \end{equation*}
    Since in this lemma we consider only outcomes such that $\Wass_2(\mu_t^\sigma; \delta_a) \leq \kappa$ and $|X^\sigma_{\tau_m} - Y^\sigma_{\tau_m}| \leq \alpha$, i.e. $\omega \in A$, after integrating over the time interval $[\tau_m; \theta_{m + 1} \wedge \mathcal{T}]$ and applying Assumption \ref{assu:pot:VetF2} (see also Definition~\ref{def:Loc_convexity}), we get the following estimate. For any $t > 0$ and for $\Prob$--a.e. $\omega \in A \cap \{t \in [\tau_{m}; \theta_{m + 1} \wedge \mathcal{T}]\}$:
    \begin{equation} \label{eq:aux:xi(t)_bound}
    \begin{aligned}
        \xi(t) & \leq |X^\sigma_{\tau_m} - Y^\sigma_{\tau_m}|^2 - 2\int_{\tau_m}^{t} \langle X^\sigma_s - Y^\sigma_s; \, \nabla W_a(X^\sigma_s) - \nabla W_a(Y^\sigma_s) \rangle \dd{s}\\
        & \quad + 2\int_{\tau_m}^{t}| X^\sigma_s - Y^\sigma_s| |\nabla F \ast \mu_s^\sigma(X^\sigma_s) - \nabla F \ast \delta_a (X^\sigma_s) | \dd{s} \\
        & \leq \alpha^2 - 2 C_{W}\int_{\tau_m}^{t}\xi(s)\dd{s}  + 2 \int_{\tau_m}^{t} \sqrt{\xi(s)} \Big|\nabla F \ast \mu_s^\sigma(X^\sigma_s) - \nabla F \ast \delta_a (X^\sigma_s) \Big| \dd{s}.
    \end{aligned}
    \end{equation}

    Since the term $\Big|\nabla F \ast \mu_s^\sigma(X^\sigma_s) - \nabla F \ast \delta_a (X^\sigma_s) \Big|$ is hard to analyse, we study it separately.

    \noindent\noindent{}{\bf Step 2.} Consider the following inequality. By Assumption (F -- 4) of \ref{assu:pot:F}, we can express:
    \begin{equation*}
    \begin{aligned}
        &\int_{\R^d} \Big| \nabla F (X^\sigma_s - z) - \nabla F (X^\sigma_s - a) \Big| \mu_s^\sigma(\dd{z}) \\
        & \leq C^\prime\int_{\R^d} |z - a| \big( 1 + \left|X_s^\sigma - z\right|^{2 r - 1} + \left|X_s^\sigma - a\right|^{2 r - 1} \big)\mu_s^\sigma(\dd{z}) \\
        &\leq C^\prime\int_{\R^d} |z - a| \big( 1 + 2^{2r - 1} \left|X_s^\sigma\right|^{2 r - 1} + 2^{2r - 2} |z|^{2 r - 1} + 2^{2r - 2} |a|^{2 r - 1} \big)\mu_s^\sigma(\dd{z}).
    \end{aligned}
    \end{equation*}

    In the following, we will denote by $\text{C}$ the generic constant that may depend on $r$, $\rho$ and other parameters defined in assumptions. The bound thus takes the form:    
    \begin{equation*}
        \begin{aligned}
            &\text{C}\int_{\R^d} |z - a| \big(\text{C} + \text{C} \left|X_s^\sigma\right|^{2 r - 1} + |z|^{2r - 1} + |a|^{2r - 1} \big)\mu_s^\sigma(\dd{z}) \\
            &\leq \text{C} \sqrt{\int_{\R^d} |z - a|^2\mu_s^\sigma(\dd{z})} \sqrt{\text{C} + \text{C} \left|X^\sigma_s \right|^{4 r - 2} + |a|^{4r - 2} + \int_{\R^d}|z|^{4r - 2}\mu_s^\sigma(\dd{s})}
        \end{aligned}
    \end{equation*}
    
    Since we only consider $\omega \in A \cap \{t \in [\tau_{m}; \theta_{m + 1} \wedge \mathcal{T}]\}$, $X^\sigma$ belongs to $B_{\rho}(a)$ and is thus bounded by a constant.  Moreover, $\Wass_2(\mu_s^\sigma; \delta_a) \leq \kappa$ and $|Y_t^\sigma| \leq \sup_{z \in \partial \cDc_R^\mathsf{e}} |z - a|$ by the definition of the set $A$. At the same time, by Proposition \ref{prop:existence}, we know that $\int|z|^{4r - 2}\dd\mu_s^\sigma \leq M$ for any time $t \geq 0$ and for any $0 \leq \sigma \leq 1$. Therefore, for any $t > 0$ and for any $\omega \in A \cap \{t \in [\tau_{m}; \theta_{m + 1} \wedge \mathcal{T}]\}$ we have
    \begin{equation*}
        \int_{\R^d} \Big| \nabla F (X^\sigma_s - z) - \nabla F (X^\sigma_s - a) \Big| \mu_s^\sigma(\dd{z}) \leq \text{C} \kappa.
    \end{equation*}
    



    \noindent\noindent{}{\bf Step 3.} Let us come back to equation \eqref{eq:aux:xi(t)_bound}. Given the calculations in Step~2, the final bound takes the following form:
    \begin{equation*}
        \xi(t) \leq \alpha^2 - 2 C_{W}\int_{\tau_m}^{t}\xi(s)\dd{s} + 2 \kappa \text{C} \int_{\tau_m}^{t} \!\!\! \sqrt{\xi(s)} \dd{s}.
    \end{equation*}
    
    It means that, if we introduce the deterministic function $\psi$ that is the unique solution of equation
    \begin{equation*}
        \psi(u) = \alpha^2 - 2 C_{W}\int_{0}^{u}\psi(s)\dd{s} + 2 \kappa \text{C} \int_{0}^{u} \sqrt{\psi(s)} \dd{s},
    \end{equation*}
    then $\xi_{\tau_m + u} \leq \psi_u$ for any positive $u \leq t$ and for $\Prob$--a.e. point $\omega \in A \cap \{t \in [\tau_{m}; \theta_{m + 1} \wedge \mathcal{T}]\}$. 
    
    If $\alpha > \frac{\text{C} \kappa}{2C_W} $, we can solve this equation explicitly and get:
    \begin{equation}\label{eq:aux:psi_t_bound}
        \sqrt{\psi(u)} = \left(\alpha - \frac{\text{C} }{2 C_{W}}\kappa\right)\e^{- 2 C_{W} u} + \frac{\text{C} }{2 C_{W}}\kappa.
    \end{equation}
    Otherwise, we can simply bound $\psi(u)$ by
    \begin{equation}\label{eq:aux:psi_t_bound_2}
        \psi(u) \leq \frac{\text{C}}{4 C_{W}^2} \kappa^2,
    \end{equation}        
    since $\psi^\prime(u) < 0$ whenever $\psi(u) > \text{C} \kappa^2/(4 C_{W}^2)$. Thus, $\psi$ can be expressed in the form:
    \begin{equation*}
        \sqrt{\psi(u)} \leq \alpha \e^{- 2 C_{W} u} + o_\kappa(1).
    \end{equation*}

    In particular, it means that if there is some random time $\mathcal{S}$ defined for $\omega \in A$ and such that for $\Prob$--a.e. $\omega \in A$ we have $\tau_m \leq \mathcal{S} \leq \theta_{m + 1} \wedge \mathcal{T}$, then:
    \begin{equation*}
        \sqrt{\xi(\mathcal{S})} \leq \alpha \e^{- 2 C_{W} (\mathcal{S} - \tau_m)} + o_\kappa(1)
    \end{equation*}
    for $\Prob$-a.e. $\omega \in A$.

    \noindent\noindent{}{\bf Step 4.} To finalise the proof, let us show that for $\Prob$--a.e. $\omega \in A$, we have $\mathcal{T} > \theta_{m + 1}$. Indeed, if it is not true, then there exists a set $B \subseteq A$ with $\Prob(B) > 0$, such that for any $\omega \in B$, $X_{\mathcal{T}}^\sigma \notin B_{\rho}(a)$, but $Y_{\mathcal{T}}^\sigma \in B_{\rho/2}(a)$. Yet, by derivations of Step~3, for $\Prob$--a.e. $\omega \in A$:
    \begin{equation*}
        |X_{\mathcal{T}}^\sigma - Y_{\mathcal{T}}^\sigma| \leq \max \left(\alpha; \frac{\text{C}}{2 C_W} \kappa \right).        
    \end{equation*}
    Therefore, without loss of generality, we can choose $\kappa > 0$ to be small enough to get the contradiction. That proves the lemma.
    
\end{proof}

For control outside of the set $B_{\rho}(a)$ consider the following lemma.

\begin{lem}\label{lem:|X-Y|_control_outside}
    Define for some constant $L > 0$ and for any $T > 0$ the following mapping: $\psi_T: x \mapsto x\e^{L T}$. Then, under Assumptions \ref{assu:pot:V}--\ref{assu:stable}, there exists a constant $L > 0$ such that for $\alpha < \rho/4$, for any $m \geq 1$, and for any $\kappa > 0$ small enough:
    \begin{equation*}
        \Prob \left(\sup_{t \in [\theta_m; \tau_{m}]} |X_t^\sigma - Y_t^\sigma| > \psi_{\tau_{m} - \theta_m}(\alpha), A \right) = 0,
    \end{equation*} 
    where $A := \{\tau_{m} \leq S_{\mathsf{st}}^\sigma(\kappa), \; \sup_{t \leq \theta_m} |X_t^\sigma - Y_t^\sigma| \leq \alpha \}$
\end{lem}

\begin{proof}
    As in the proof of Lemma \ref{lem:|X-Y|_control_inside}, we first introduce $\xi(t) = |X_t^\sigma - Y_t^\sigma|^2$ and then differentiate this function with respect to time. The difference is that now we can not use convexity properties of the set $B_{\rho}(a)$. Moreover, we will not be able to provide a good upper bound for $|X^\sigma_t|$, since $Y^\sigma$ and $X^\sigma$ drift apart from each other. 
    
    \noindent\noindent{}{\bf Step 1.} The following inequality holds for $\Prob$--a.e. $\omega \in A \wedge \{t \in [\theta_m; \tau_m]\}$:
    \begin{equation*}
    \begin{aligned}
        \xi(t) & \leq |X_{\theta_m}^\sigma - Y_{\theta_m}^\sigma|^2 - 2 \int_{\theta_m}^t \langle X_s^\sigma - Y^\sigma_s; \nabla V(X_s^\sigma) - \nabla V (Y_s^\sigma) \rangle \dd{s} \\
        & \quad - 2 \int_{\theta_m}^t \langle X_s^\sigma - Y_s^\sigma; \nabla F \ast \mu_s^\sigma (X_s^\sigma) - \nabla F(Y^\sigma_s - a) \rangle \dd{s}.
    \end{aligned}
    \end{equation*}
    
    Using Cauchy–Schwarz inequality, we can obtain the following bound:
    \begin{equation*}
    \begin{aligned}
        \xi(t) & \leq \alpha^2 + 2 \int_{\theta_m}^t \sqrt{\xi(s)} \;\big|\nabla V(X_s^\sigma) - \nabla V (Y_s^\sigma) \big| \dd{s} \\
        & \quad + 2 \int_{\theta_m}^t \sqrt{\xi(s)} \; \big| \nabla F \ast \mu_s^\sigma (X_s^\sigma) - \nabla F(Y^\sigma_s - a) \big| \dd{s} =: \alpha^2 + I_1 + I_2.
    \end{aligned}
    \end{equation*}
    Let us consider $I_1$ and $I_2$ separately. In the following, $\text{C}$ will denote a generic constant that may depend on parameters defined in the assumptions.

    \noindent\noindent{}{\bf Step 2.} For the first expression $I_1$, we use Assumption (V -- 5) of \ref{assu:pot:V} and get:
    \begin{equation*}
        2\! \int_{\theta_m}^t \!\!\!\! \sqrt{\xi(s)} \;\big|\nabla V(X_s^\sigma) - \nabla V (Y_s^\sigma) \big| \dd{s} \leq \text{C} \! \! \int_{\theta_m}^t \!\!\! \xi(s) \! \left(1 + |X_s^\sigma|^{2r - 1} + |Y_s^\sigma|^{2r - 1} \right) \dd{s}.
    \end{equation*}

    By adding and subtracting $Y^\sigma_s$ in the expression above, we can upper bound it by
    \begin{equation*}
        \text{C} \int_{\theta_m}^t \xi(s) \left(\text{C} + \text{C}\xi(s)^{\frac{2r - 1}{2}} + |Y_s^\sigma|^{2r - 1} \right) \dd{s}. 
    \end{equation*}
    Moreover, since we consider only those $\omega$ for which $t \leq \tau_m$, $Y_s^\sigma$ belongs to $\cDc^\mathsf{e}_R$, which is a bounded set. Therefore, the upper bound takes the final form:
    \begin{equation}\label{eq:aux:I_1}
        I_1 \leq \text{C} \int_{\theta_m}^t \xi(s) \left(\text{C} + \xi(s)^{\frac{2r - 1}{2}}\right) \dd{s}.
    \end{equation}

    \noindent\noindent{}{\bf Step 3.} For the second expression $I_2$, let us use assumption (F -- 4) of \ref{assu:pot:F} and get:
    \begin{equation*}
    \begin{aligned}
        I_2 &\leq 
        \text{C} \int_{\theta_m}^t \sqrt{\xi(s)} \int_{\R^d} |X_s^\sigma - z - Y^\sigma_s + a| \\
        &\quad \times \big( 1 + \left|X_s^\sigma - z\right|^{2 r - 1} + \left|Y_s^\sigma - a\right|^{2 r - 1} \big)\mu_s^\sigma(\dd{z})\dd{s} \\ 
        & \leq \text{C} \int_{\theta_m}^t \int_{\R^d} \xi(s) \big( 1 + \left|X_s^\sigma - z\right|^{2 r - 1} + \left|Y_s^\sigma - a\right|^{2 r - 1} \big)\mu_s^\sigma(\dd{z})\dd{s} \\
        & \quad + \text{C} \int_{\theta_m}^t \int_{\R^d} \sqrt{\xi(s)} |z - a| \big( 1 + \left|X_s^\sigma - z\right|^{2 r - 1} + \left|Y_s^\sigma - a\right|^{2 r - 1} \big)\mu_s^\sigma(\dd{z})\dd{s}.
    \end{aligned}
    \end{equation*}

    Let us denote the two expressions above as $A_1$ and $A_2$. For $A_1$, we add and subtract $Y_s^\sigma$ inside $|X_t^\sigma - z|^{2r - 1}$ and get:
    \begin{equation*}
        A_1 \leq \text{C} \int_{\theta_m}^t \int_{\R^d} \xi(s) \left( \text{C} + \text{C} \xi(s)^{\frac{2r - 1}{2}} + \text{C}|Y_s^\sigma|^{2r - 1} + |z|^{2r - 1} + |a|^{2r - 1}\right) \mu_s^\sigma(\dd{z}) \dd{s}.
    \end{equation*}
    As was pointed out above, since $t \in [\theta_m; \tau_m]$, $|Y_s^\sigma - a|$ is bounded for $\Prob$--a.e. $\omega \in A \wedge \{t \in [\theta_m; \tau_m]\}$. Moreover, by Proposition \ref{prop:existence}, there exists $M > 0$ such that $\int |z|^{2r - 1} \dd\mu_s^\sigma < M$. Thus:
    \begin{equation*}
        A_1 \leq \text{C} \int_{\theta_m}^t \xi(s) \left( \text{C} + \xi(s)^{\frac{2r - 1}{2}} \right) \dd{s}.
    \end{equation*}
    Similarly, for $A_2$:
    \begin{equation*}
    \begin{aligned}
        A_2 &\leq \text{C} \int_{\theta_{m}}^t \int_{\R^d} \sqrt{\xi(s)} |z - a| \\
        & \quad \times \left( \text{C} + \text{C} \xi(s)^{\frac{2r - 1}{2}} + \text{C}|Y_s^\sigma|^{2r - 1} + |z|^{2r - 1} + |a|^{2r - 1} \right) \mu_s^\sigma(\dd{z}) \dd{s}.
    \end{aligned}
    \end{equation*}
    By Cauchy–Schwarz inequality and since both $\int_{\R^d}|z| \dd \mu_s^\sigma$ and $\int_{\R^d}|z|^{4r - 2}\dd\mu_s^\sigma$ are bounded by a constant, we get:
    \begin{equation*}
        A_2 \leq \text{C} \int_{\theta_{m}}^t \sqrt{\xi(s)} \sqrt{\text{C} + \xi(s)^{2r - 1}} \dd{s},
    \end{equation*}
    which gives the following bound for $I_2$:
    \begin{equation*}
        I_2 \leq \text{C} \int_{\theta_m}^t \Bigg( \text{C} \xi(s) \left( \text{C} + \xi(s)^{\frac{2r - 1}{2}} \right) + \sqrt{\xi(s)} \sqrt{\text{C} + \xi(s)^{2r - 1}} \Bigg)\dd{s}.
    \end{equation*}
    Since for any $\alpha > 1$ we have $x^\alpha \leq \sqrt{x} + x^{\alpha + 1}$ and since $\sqrt{x} \leq 1 + x$, we can roughly bound $I_2$ by the following expression:
    \begin{equation}\label{eq:aux:I_2}
        I_2 \leq \text{C} \int_{\theta_m}^t \left(\xi(s)^{r + 1} + \sqrt{\xi(s)} \right)\dd{s}
    \end{equation}

    \noindent\noindent{}{\bf Step 4.} From \eqref{eq:aux:I_1} and \eqref{eq:aux:I_2} we get that for $\Prob$--a.e. $\omega \in A \wedge \{t \in [\theta_m; \tau_m]\}$:
    \begin{equation*}
        \xi(t) \leq  \alpha^2 + \text{C} \int_{\theta_m}^t \left(\text{C}\xi(s)^{r + 1} + \sqrt{\xi(s)}\right)\dd{s}.
    \end{equation*}

    Obviously, $\xi(t)$ is bounded for respective $\omega$ and $t$ by a function of the form:
    \begin{equation*}
        \psi(u) = \alpha^2 + \text{C} \int_{0}^u \left(\text{C}\psi(s)^{r + 1} + \sqrt{\psi(s)}\right)\dd{s},
    \end{equation*}
    which in its term is bounded by the following expression. Note that for each period of time when $\psi(u) \leq 1$, it is simply bounded by a linear function:
    \begin{equation*}
        \psi(u) \leq \alpha^2 + \text{C} u.
    \end{equation*}
    Otherwise, its upper bound take the form:
    \begin{equation*}
        \psi(u) \leq \alpha^2 + \text{C} \int_{0}^u \psi(s)^{r + 1}\dd{s},
    \end{equation*}
    which is a polynomial. By choosing the right constant $L > 0$, we can easily bound $\psi$ by
    \begin{equation*}
        \psi(u) \leq \alpha^2\e^{2Lu},
    \end{equation*}
    which proves the Lemma by using the same approach as in Steps 3 and 4 of the proof of Lemma \ref{lem:|X-Y|_control_inside}.
\end{proof}

The following lemma establishes the maximum number of excursions of the process $Y^\sigma$ from $B_\rho(a)$. Let us define the height of the effective potential inside the sets of the form $B_{\rho/2}(a)$ as $Q^\mathsf{c} := \inf_{z \in S_{\rho/2}(a)} \{W_a(z) - W_a(a)\}$. We remind that $H^\mathsf{e}_R := \inf_{z \in \partial \cDc_{R}^\mathsf{e}} \{W_a(z) - W_a(a)\}$ is the height of the effective potential inside the set $\cDc^\mathsf{e}_R$. Consider the following lemma:
\begin{lem}\label{lem:N_star_and_T_1}
    Let $N^{*} := 2 \left\lceil \exp{\frac{2}{\sigma^2} (H_R^\mathsf{e} - Q^{\mathsf{c}} + \kappa)} \right\rceil$. Let $\tau_{N^*}$ be defined as in \eqref{eq:def_of_tau_theta}. Then, for any $\kappa > 0$ small enough:
    \begin{enumerate}
        \item $\Prob(\tau_{\cDc_R^\mathsf{e}}^{Y, \sigma} > \tau_{N^*}) \xrightarrow[\sigma \to 0]{} 0$.
        \item There exists $T_1 > 0$ such that $\Prob(\exists i \leq N^*: \theta_i - \tau_{i - 1
        } > T_1) \xrightarrow[\sigma \to 0]{} 0$.
    \end{enumerate}
    
\end{lem}

\begin{proof}
We separate the proof into 2 steps.

    \noindent\noindent{}{\bf Step 1.} Let us prove the first part of the lemma. Note, that if $\tau_{N^*}$ is less or equal then $\exp{\frac{2}{\sigma^2}(H^\mathsf{e}_R + \frac{\kappa}{2})}$, then necessarily the number of intervals of the form $[\tau_{i - 1}; \theta_i]$ such that $\theta_i - \tau_{i - 1} \geq \exp{\frac{2}{\sigma^2}(Q^\mathsf{c} - \frac{\kappa}{2})}$ can not exceed $N^*/2$ by definition of the latter. Based on this observation and using Proposition~\ref{prop:DZ_exit_time}, we have 
    \begin{equation}\label{eq:aux:num_of_i_leq_N_star}
        \begin{aligned}
        \Prob(\tau_{\cDc_R^\mathsf{e}}^{Y, \sigma} > \tau_{N^\ast}) &\leq \Prob\Big(\tau_{N^\ast} < \tau_{\cDc_R^\mathsf{e}}^{Y, \sigma} < \e^{\frac{2}{\sigma^2} (H_R^\mathsf{e} + \frac{\kappa}{2})}\Big) + \Prob \left(\tau_{\cDc_R^\mathsf{e}}^{Y, \sigma} \geq \e^{\frac{2}{\sigma^2}(H_R^\mathsf{e} + \frac{\kappa}{2})} \right) \\
        & \leq \Prob\Big(\# \Big\{i \leq N^*: \theta_i - \tau_{i - 1} < \e^{\frac{2}{\sigma^2} (Q^\mathsf{c} - \frac{\kappa}{2})} \Big\} > \frac{N^*}{2} \Big) + o_\sigma(1),
        \end{aligned}    
    \end{equation}
    where $o_\sigma(1)$ is an infinitesimal with respect to $\sigma$. Consider 
    \begin{equation}\label{eq:aux:num_of_i_leq_N_star2}
        \begin{aligned}
            \Prob\Big(\# & \Big\{i \leq N^*: \theta_i - \tau_{i - 1} < \e^{\frac{2}{\sigma^2} (Q^\mathsf{c} - \frac{\kappa}{2})} \Big\} > \frac{N^*}{2} \Big) \\
            &\leq \sum_{k = \left\lceil \frac{N^*}{2}\right\rceil}^{N^*} \sum_{(i_1, \,\dots, i_k)} \Prob \left(\bigcap_{j = 1}^k \Big\{\theta_{i_j} - \tau_{i_j - 1} < \e^{\frac{2}{\sigma^2}(Q^\mathsf{c} - \frac{\kappa}{2})} \Big\} \right) \\
            &\leq \sum_{k = \left\lceil \frac{N^*}{2}\right\rceil}^{N^*} 2^{N^*} \Big(\sup_{y \in B_{\rho/4}(a)}\TrProb_{y}\!\left(\theta_0 < \e^{\frac{2}{\sigma^2}(Q^\mathsf{c} - \frac{\kappa}{2})} \right) \Big)^k,
        \end{aligned}
    \end{equation}
    where $(i_1,\; \dots \;, i_k)$ stands for all possible choices of $k$ numbers $i_1 < \dots < i_k$ from the set $\{1, \; \dots \;, N^*\}$. Note that the number of such combinations can be roughly bounded by $2^{N^*}\!$. The last inequality in \eqref{eq:aux:num_of_i_leq_N_star2} we get due to the fact that $Y^\sigma$ is a strong Markov process and $\theta_0$ is defined in \eqref{eq:def_of_tau_0_theta_0}. By the exit-time result for diffusions of type $Y^\sigma$ (see Proposition~\ref{prop:DZ_exit_time}), for any $\kappa < \rho/4$:
    \begin{equation*}
        \sup_{y \in B_{\rho/4}(a)}\TrProb_{y}\!\left(\theta_0 < \e^{\frac{2}{\sigma^2}(Q^\mathsf{c} - \frac{\kappa}{2})} \right) = o_{\sigma}(1).
    \end{equation*}
    After adding this bound to equations \eqref{eq:aux:num_of_i_leq_N_star2} and \eqref{eq:aux:num_of_i_leq_N_star}, we get:
    \begin{equation*}
        \begin{aligned}
            \Prob(\tau_{\cDc_R^\mathsf{e}}^{Y, \sigma} > \tau_{N^\ast}) & \leq 2^{N^*} o_\sigma(1)^{\left\lceil \frac{N^*}{2}\right\rceil}\frac{1 - o_\sigma(1)^{\left\lceil \frac{N^*}{2}\right\rceil}}{1 - o_{\sigma}} + o_\sigma(1) = o_\sigma(1).
        \end{aligned}
    \end{equation*}

    \noindent\noindent{}{\bf Step 2.} For the second part of the lemma, we use \cite[Lemma 5.7.19]{DZ}, that is the fact that there exists $T_1 > 0$ big enough such that
    \begin{equation}\label{eq:aux:tau_0>T_1}
        \limsup_{\sigma \to 0} \frac{\sigma^2}{2} \log \sup_{y \in B_{\rho/2}(a)} \TrProb_y (\tau_0 > T_1) < - (H_R^{\mathsf{e}} - Q^\mathsf{c} + 1).
    \end{equation}
    
    Consider the following equations:    
    \begin{equation*}
        \begin{aligned}
            \Prob(\exists i \leq N^*: \tau_{i} - \theta_i  > T_1) &\leq \sum_{i = 1}^{N^*} \Prob(\tau_i - \theta_i > T_1 ) \leq N^* \sup_{y \in B_{\rho/2}(a)}\TrProb_{y} (\tau_0 > T_1),
        \end{aligned}
    \end{equation*}
    where the last inequality is due to the Markov property of the diffusion $Y^\sigma$. Finally, by \eqref{eq:aux:tau_0>T_1}, we get:
    \begin{equation*}
    \begin{aligned}
        \Prob(\exists i \leq N^*: \tau_{i} - \theta_i  > T_1) &\leq  2 \left(\e^{2(H_R^\mathsf{e} - Q^{\mathsf{c}} + \kappa)/\sigma^2} + 1\right)\e^{-2(H_R^\mathsf{e} - Q^{\mathsf{c}} + 1)/\sigma^2} \\
        & \xrightarrow[\sigma \to 0]{} 0,
    \end{aligned}
    \end{equation*}
    which proves the lemma if $\kappa$ is chosen to be small enough.
\end{proof}

Now we are ready to prove Proposition~\ref{prop:|X_t-Y_t|}.

\begin{proof}[Proof of Proposition \ref{prop:|X_t-Y_t|}]
    Since, by Lemma \ref{lem:N_star_and_T_1}, each time spent outside of $B_{\rho/2}(a)$ is bounded by a constant $T_1 > 0$ with high probability, we are interested in the composition
    \begin{equation}
        \psi_{T_1} \circ \phi_{t} (x) = x \e^{(LT_1 - Kt)} + o_\kappa(1) \e^{LT_1} \leq x \e^{(LT_1 - Kt)} + o_\kappa(1).
    \end{equation}
    Let us introduce the following mapping:
    \begin{equation*}
        \Psi_t(x) := x \e^{L T_1 - Kt} + o_\kappa(1).
    \end{equation*}

    Then the results of Lemmas \ref{lem:|X-Y|_control_inside} and \ref{lem:|X-Y|_control_outside} can be rewritten in the following form: for any $\kappa > 0$ small enough, for any $m \geq 1$ and for any $\alpha < \kappa$:
    \begin{equation}\label{eq:Psi_t_|X_t-Y_t|}
    \begin{aligned}
        \Prob \Bigg(\sup_{t \in [\tau_m; \tau_{m + 1}]} |X_t^\sigma - Y_t^\sigma|& > \Psi_{\theta_{m + 1} - \tau_m}(\alpha);  \\
        &\quad \tau_{m + 1} \leq S_{\mathsf{st}}^\sigma(\kappa), \sup_{t \leq \tau_m} |X_t^\sigma - Y_t^\sigma| \leq \alpha \Bigg) = 0.
    \end{aligned}
    \end{equation}

    Let us now come back to the statement of the proposition. Fix some $0 < \eta < H_R^\mathsf{e} - H$. Note that,  if $\sup_t |X_t^\sigma - Y_t^\sigma| > \alpha$, for $t \in [T_{\mathsf{st}}^\sigma(\kappa); S_\mathsf{st}^\sigma(\kappa) \wedge \e^{\frac{2}{\sigma^2}(H + \eta)} ]$, then it should happen for $t$ belonging to one of the periods of time of the form $[\tau_{k - 1}; \tau_{k}]$ that are before $S_\mathsf{st}^\sigma(\kappa) \wedge \e^{\frac{2}{\sigma^2}(H + \eta)}$. Moreover, since we know, by Lemma \ref{lem:N_star_and_T_1}, that $\tau_{N^*}$ happens after $S_\mathsf{st}^\sigma(\kappa) \wedge \e^{\frac{2}{\sigma^2}(H + \eta)}$ with high probability, the number of periods of the form $[\tau_{k - 1}; \tau_{k}]$, during which $|X_t^\sigma - Y_t^\sigma|$ can surpass the level $\alpha$, is bounded by $N^*$. Given these observations, consider the following line of equations:
     \begin{equation}\label{eq:aux:|X_t-Y_t|>kappa}
     \begin{aligned}
        &\Prob \left(\sup \Big\{\big|X_t^\sigma - Y_t^\sigma \big|: t \in \Big[T_{\mathsf{st}}^\sigma(\kappa); S_\mathsf{st}^\sigma(\kappa) \wedge \exp{\frac{2}{\sigma^2}(H + \eta)} \Big] \Big\} > \alpha \right) \\
        & \leq \Prob \left( \tau_{N^*} \leq S_{\mathsf{st}}^\sigma(\kappa) \wedge \exp{\frac{2}{\sigma^2}(H + \eta)} \right) + \Prob(\exists m \leq N^*: \tau_m - \theta_m > T_1) \\
        &\quad  + \Prob \Big(\exists k^* \leq N^*: \sup_{t \in [\tau_{k^* - 1}; \tau_{k^*}]} |X_t^\sigma - Y_t^\sigma| > \alpha \text{ and } \\
        & \qquad\qquad  \tau_{k^*} \leq S_{\mathsf{st}}^\sigma(\kappa) \wedge \exp{\frac{2}{\sigma^2}(H + \eta)},  \forall m \leq N^*: \tau_m - \theta_m \leq T_1,  \Big) \\
        & = : I_1 + I_2 + I_3.
     \end{aligned}
    \end{equation}

    For the first probability:
    \begin{equation*}
    \begin{aligned}
        I_1 &\leq \Prob \left(\tau_{N^*} \leq \tau_{\cDc_R^{\mathsf{e}}}^{Y, \sigma} \right) + \Prob \left(\tau_{\cDc_R^{\mathsf{e}}}^{Y, \sigma} < \tau_{N^*}  \leq \exp{\frac{2}{\sigma^2}(H + \eta)} \right)\\
        &\leq \Prob \left(\tau_{N^*} \leq \tau_{\cDc_R^{\mathsf{e}}}^{Y, \sigma} \right) + \Prob(|X_{T_{\mathsf{st}}^\sigma(\kappa)} - a| > \kappa) + \sup_{y \in B_{\kappa}(a)} \!\!\!\TrProb_y(\tau_{\cDc_R^{\mathsf{e}}}^{Y, \sigma} < \exp{\frac{2}{\sigma^2}(H + \eta)}) \\
        &\xrightarrow[\sigma \to 0]{} 0,         
    \end{aligned}
    \end{equation*}
    by Lemmas \ref{lem:N_star_and_T_1}, \ref{lem:stab_in_finite_time} and Proposition \ref{prop:DZ_exit_time}, and since $H + \eta < H_{R}^\mathsf{e}$. At the same time, by Lemma \ref{lem:N_star_and_T_1}, the second expression:
    \begin{equation*}
        I_2 \xrightarrow[\sigma \to 0]{} 0.
    \end{equation*}
    
    What is left is the third expression. Note that, by \eqref{eq:Psi_t_|X_t-Y_t|}, $I_3$ is bounded by:
    \begin{equation*}
        \sum_{k^* = 1}^{N^*} \Prob(\Psi_{\theta_{k^*} - \tau_{k^* - 1}} \circ \dots \circ \Psi_{\theta_2 - \tau_1} \circ \Psi_{\theta_1 - T_{\mathsf{st}}(\kappa)}\circ 0 > \kappa).         
    \end{equation*}
    We get that expression by observing that, if there exists $k^*$ such that inequality $$\sup_{t \in [\tau_{k^*}; \tau_{k^* + 1}]} |X_t^\sigma - Y_t^\sigma| > \kappa$$ holds, then, given that this difference is smaller than $\kappa$ for times smaller than $\tau_{k^*}$, we can control this difference in terms of $\Psi_{\theta_k - \tau_{k - 1}}$ by \eqref{eq:Psi_t_|X_t-Y_t|}. 

    Let us study the sum above. By definition of $\Psi_T$, we have:    
    \begin{equation*}
    \begin{aligned}
        \sum_{k^* = 1}^{N^*}&\Prob(\Psi_{\theta_{k^*} - \tau_{k^* - 1}} \circ \dots \circ \Psi_{\theta_2 - \tau_1} \circ \Psi_{\theta_1 - T_{\mathsf{st}}(\kappa)}\circ 0 > \kappa) \\        
        & \leq \sum_{k^* = 1}^{N^*} \Prob \left(\sum_{i = 1}^{k^* - 1}o^i_\kappa(1)\exp{\sum_{j = i}^{k^*}(L T_1 - K(\theta_{j} - \tau_{j - 1}))} + o^{k^*}_\kappa(1)  > \kappa \right). 
    \end{aligned}
    \end{equation*}
We can continue the calculations and get the following upper bound:    
    \begin{equation*}
    \begin{aligned}
        &\sum_{k^* = 1}^{N^*} \Prob \left(\sup_{1 \leq i \leq k^* - 1}\exp{\sum_{j = i}^{k^*}(L T_1 - K(\theta_{j} - \tau_{j - 1}))}  > \frac{\kappa - o^{k^*}_\kappa(1)}{k^* - 1} \right) \\
        & \leq \sum_{k^* = 1}^{N^*} \Prob \left(\sup_{1 \leq i \leq k^* - 1} \left\{\sum_{j = i}^{k^*}(L T_1 - K(\theta_{j} - \tau_{j - 1})) \right\}  > - \log (k^* - 1) \right)
    \end{aligned}
    \end{equation*}

    Note that, if there are more than $(k^* - i + 1)/2$ intervals of the size $(\theta_j - \tau_{j - 1}) > \exp{\frac{2(Q^\mathsf{c} - \kappa)}{\sigma^2}}$, then necessarily
    \begin{equation*}
    \begin{aligned}
        \sum_{j = i}^{k^*} (LT_1 - K(\theta_j - \tau_{j - 1})) &\leq (k^* - i + 1) LT_1  - \left \lceil\frac{(k^* - i + 1)}{2} \right \rceil K \e^{\frac{2 (Q^\mathsf{c} - \kappa) }{\sigma^2}}\\ 
        & \leq  (k^* - i + 1) \Big(LT_1 - \frac{K}{2}\e^{\frac{2 (Q^\mathsf{c} - \kappa)}{\sigma^2}} \Big)\\
        &\leq k^* \Big(LT_1 - \frac{K}{2}\e^{\frac{2 (Q^\mathsf{c} - \kappa)}{\sigma^2}} \Big).      
    \end{aligned}
    \end{equation*}
    Since $k^* > \log(k^* - 1)$ for any $k^* \geq 1$ and since $\Big(LT_1 - \frac{K}{2}\e^{\frac{2 (Q^\mathsf{c} - \kappa)}{\sigma^2}} \Big)$ is negative for small enough $\sigma$, we get
    \begin{equation*}
        \sup_{1 \leq i \leq k^* - 1} \left\{\sum_{j = i}^{k^*}(L T_1 - K(\theta_{j} - \tau_{j - 1})) \right\}  \leq - \log (k^* - 1),
    \end{equation*}
    which means that it is impossible to have more than $(k^* - i + 1)/2$ intervals of the size $\theta_j - \tau_{j - 1} \geq \exp{\frac{2(Q^\mathsf{c} - \kappa)}{\sigma^2} }$. Therefore, we have:
    
    \begin{equation*}
        \begin{aligned}
            \Prob &\left(\sup_{1 \leq i \leq k^* - 1} \Big\{\sum_{j = i}^{k^*}(L T_1 - K(\theta_{j} - \tau_{j - 1})\} \Big)  > - \log (k^* - 1) \right)\\
            & \leq \Prob \Bigg( \forall i \leq k^* - 1: \# \left\{j: i \leq j \leq k^*: \theta_j - \tau_{j - 1} \leq \exp{\frac{2(Q^\mathsf{c} - \kappa)}{\sigma^2} } \right\} \\
            & \qquad\qquad\qquad\qquad\qquad\qquad \geq \frac{k^* - i + 1}{2}\Bigg) \\
            & \leq \min_{1 \leq i \leq k^* - 1}\sum_{n = \left\lfloor \frac{k^* - i + 1}{2} \right\rfloor}^{k^*} \sum_{(j_1, \dots, j_{n})} \Prob \left(\bigcap_{l = 1}^{n}\left\{ \theta_{j_l} - \tau_{j_l - 1}\leq \exp{\frac{2(Q^\mathsf{c} - \kappa)}{\sigma^2} } \right\} \right).
        \end{aligned}
    \end{equation*}    
    Since the number of combinations of the form $(j_1, \dots, j_n)$ can be roughly bounded by $2^n$, we can deduce
    \begin{equation*}
    \begin{aligned}
       \min_{1 \leq i \leq k^* - 1} & \sum_{n = \left\lfloor \frac{k^* - i + 1}{2} \right\rfloor}^{k^*}\sum_{(j_1, \dots, j_{n})} \Prob \left(\bigcap_{l = 1}^{n}\left\{ \theta_{j_l} - \tau_{j_l - 1}\leq \exp{\frac{2(Q^\mathsf{c} - \kappa)}{\sigma^2} } \right\} \right) \\
        & \leq \sum_{n = \left\lfloor \frac{k^*}{2} \right\rfloor}^{k^*} 2^{n} \left( \sup_{y \in B_{\rho/4}(a)}\TrProb_{y}\! \left(\theta_0 \leq \exp{\frac{2(Q^\mathsf{c} - \kappa)}{\sigma^2} }\right) \right)^{n} \\
        & \leq o_\sigma(1)^{\left\lfloor \frac{k^* + 1}{2} \right\rfloor} \frac{1 + o_\sigma(1)^{\left\lfloor \frac{k^* + 1}{2} \right\rfloor}}{1 - o_\sigma(1)},
    \end{aligned}
    \end{equation*}
    by Proposition~\ref{prop:DZ_exit_time}.

    Combining inequalities above, we can come back to \eqref{eq:aux:|X_t-Y_t|>kappa} and conclude that $I_3$ also tends to $0$ with $\sigma \to 0$, for each $\kappa > 0$ small enough, which finalizes the proof.
    
\end{proof}

\subsection{Control of \texorpdfstring{$Y^\sigma$}{Y}: Proof of Lemma \ref{lem:Y_t_notin_B_rho}}
\label{s:proof_control_of_Y_sigma}

We can show, using large deviations techniques, that there exists a uniform upper bound on the time of convergence of $Y^\sigma$ inside $B_{\rho/4}(a)$. Namely, for any $r>0$ small enough, there exists $\overline{T} > 0$ such that
\begin{equation*}
    \sup_{y \in \cDc_{r}^\mathsf{e}} \TrProb_y(Y^\sigma_{\overline{T}} \notin B_{\rho/4}(a)) \xrightarrow[\sigma \to 0]{} 0.
\end{equation*}
The construction of such a $\overline{T}$ can be found in \cite[Proof of Lemma 5.7.19]{DZ}. 

Therefore, for small enough $\sigma > 0$, given only $r$ and $\rho$, we can choose a continuous function $\overline{o}(\sigma)$ such that $\overline{o}(\sigma) \xrightarrow[\sigma \to 0]{} 0$ and we have
\begin{equation}\label{eq:aux:prob_Y_T_notin_B_a}
    \sup_{y \in \cDc_{r}^\mathsf{e}} \TrProb_y(Y^\sigma_{\overline{T}} \notin B_{\rho/4}(a)) \leq \overline{o}(\sigma),
\end{equation}
for all $\sigma > 0$ small enough.

Moreover, by Proposition~\ref{prop:DZ_exit_time}, we know that for any $\delta > 0$ we have $\Prob(\tau_{\cDc_r^\mathsf{e}}^{Y, \sigma} \leq \exp{\frac{2(H_r^\mathsf{e} - \delta)}{\sigma^2}}) \xrightarrow[\sigma \to 0]{} 0$. After fixing some positive $r>0$ and choosing $\delta$ to be small enough such that $H < H^\mathsf{e}_r - \delta$, we can define $\eta > 0$ as a small enough number such that $H + \eta < H^\mathsf{e}_r - \delta$.
In the following, we can restrict ourselves only to those trajectories that do not leave domain $\cDc_r^\mathsf{e}$ before time $\exp{\frac{2(H + \eta)}{\sigma^2}} < \exp{\frac{2(H_r^\mathsf{e} - \delta)}{\sigma^2}}$. Define the event $A := \{\tau^{Y, \sigma}_{\cDc_r^\mathsf{e}} > \exp{\frac{2(H + \eta)}{\sigma^2}}\}$. 

Consider the following inequalities. By Lemma~\ref{lem:stab_in_finite_time} and the definition of $Y^\sigma$, for any $\kappa > 0$, we can introduce $\overline{o}_\kappa(\sigma)$, the modification of function $\overline{o}(\sigma)$ such that \ref{eq:aux:prob_Y_T_notin_B_a} still holds and also we have
\begin{equation}\label{eq:aux:Y_sigma_T_st}
    \Prob(Y^\sigma_{T_\mathsf{st}^\sigma(\kappa)} \notin B_{\rho/4}(a), A) \leq \overline{o}_\kappa(\sigma).
\end{equation}
At the same time, using the Markov property of diffusion $Y^\sigma$, for small enough $\sigma > 0$, we have
\begin{equation*}
\begin{aligned}
    \Prob &(Y^\sigma_{T_\mathsf{st}^\sigma(\kappa) + \overline{T}} \notin B_{\rho/4}(a), A) \\
    &\leq  \sup_{y \in  B_{\rho/4}(a)}\TrProb_y (Y^\sigma_{\overline{T}} \notin B_{\rho/4}(a), A) \Prob(Y^\sigma_{T_\mathsf{st}^\sigma(\kappa)} \in B_{\rho/4}(a), A) \\
    & \quad +   \sup_{y \in \cDc_r^\mathsf{e} \setminus B_{\rho/4}(a)}\TrProb_y(Y^\sigma_{\overline{T}} \notin B_{\rho/4}(a), A) \Prob(Y_{T_\mathsf{st}^\sigma(\kappa)} \notin B_{\rho/4}(a), A)\\
    & \leq \overline{o}_\kappa(\sigma) + \overline{o}_\kappa^2(\sigma),
\end{aligned}
\end{equation*}
by Equations \eqref{eq:aux:prob_Y_T_notin_B_a} and \eqref{eq:aux:Y_sigma_T_st}, while $\Prob(Y^\sigma_{T_\mathsf{st}^\sigma(\kappa)} \in B_{\rho/4}(a), A)$ is bounded by 1. For the next step consider:
\begin{equation*}
\begin{aligned}
    \Prob &(Y^\sigma_{T_\mathsf{st}^\sigma(\kappa) + 2\overline{T}} \notin B_{\rho/4}(a), A) \\
    &\leq  \sup_{y \in  B_{\rho/4}(a)}\TrProb_y(Y^\sigma_{\overline{T}} \notin B_{\rho/4}(a), A) \Prob(Y^\sigma_{T_\mathsf{st}^\sigma(\kappa) + \overline{T}} \in B_{\rho/4}(a), A) \\
    & \quad +   \sup_{y \in \cDc_r^\mathsf{e} \setminus B_{\rho/4}(a)}\TrProb_y(Y^\sigma_{\overline{T}} \notin B_{\rho/4}(a), A) \Prob(Y^\sigma_{T_\mathsf{st}^\sigma(\kappa) + \overline{T}} \notin B_{\rho/4}(a), A)\\
    & \leq \overline{o}_\kappa(\sigma) \left[1 + \overline{o}_\kappa(\sigma) + \overline{o}_\kappa^2(\sigma) \right],
\end{aligned}
\end{equation*}
similarly to the previous computations. For any fixed $\kappa > 0$ and $\sigma > 0$ small enough, we can repeat this procedure $N(\sigma) := \left\lfloor \frac{1}{\overline{T}}  \exp{\frac{2(H + \eta)}{\sigma^2}} \right\rfloor$ times, thus while $A$ still holds. We finally get the following upper bound:
\begin{equation}\label{eq:aux:Y_T_st_n_T_1}
     \sup_{n \leq N(\sigma)} \Prob (Y^\sigma_{T_\mathsf{st}^\sigma(\kappa) + n \overline{T}} \notin B_{\rho/4}(a), A) \leq \overline{o}_\kappa(\sigma) \sum_{i = 0}^{N(\sigma)} \overline{o}^i_\kappa(\sigma) \leq \frac{\overline{o}_\kappa(\sigma)}{1 - \overline{o}_\kappa(\sigma)}. 
\end{equation}
This allows us to confine with high probability $Y^\sigma$ for points of time of the form $T_\mathsf{st}^\sigma(\kappa) + n \overline{T}$ inside the ball $B_{\rho/4}(a)$. 

The last steps that one has to make in order to prove the lemma is, first, to control the probability $\Prob(Y^\sigma_t \notin B_{\rho/2}(a))$ in between points of time of the form $T_\mathsf{st}^\sigma(\kappa) + k \overline{T}$ and $T_\mathsf{st}^\sigma(\kappa) + (k + 1) \overline{T}$ and, second, remove event $A$. Note that
\begin{equation*}
    \sup_t \Prob (Y_t^\sigma \notin B_{\rho/2}(a)) \leq \Prob(\overline{A}) + \sup_t \Prob (Y_t^\sigma \notin B_{\rho/2}(a), A),
\end{equation*}
where the suprema are taken with respect to $t \in \left[T_{\mathsf{st}}^\sigma(\kappa); \exp{\frac{2(H + \eta)}{\sigma^2}} \right]$. The first probability tends to zero by Lemma~\ref{lem:stab_in_finite_time} and Proposition~\ref{prop:DZ_exit_time}, since
\begin{equation*}
    \Prob(\overline{A}) \leq \Prob(|X_{T_{\mathsf{st}}^\sigma(\kappa)}^\sigma - a| > \kappa) + \sup_{y \in B_{\kappa}(a)} \TrProb_y \left(\tau^{Y, \sigma}_{\cDc_r^\mathsf{e}} < \exp{\frac{2(H_r^\mathsf{e} - \delta)}{\sigma^2}} \right) \xrightarrow[\sigma \to 0]{} 0.
\end{equation*}

For the second probability, consider the following inequalities for any $\kappa > 0$ and for any $\sigma > 0$ small enough. Using the Markov property of $Y^\sigma$, we have
\begin{equation*}
\begin{aligned}
    \sup_t \Prob (Y_t^\sigma \notin B_{\rho/2}(a), A) & \leq \sup_{n \leq N(\sigma)} \Prob(Y^\sigma_{T_{\mathsf{st}}^\sigma(\kappa) + n \overline{T}} \in B_{\rho/4}(a), A) \\
    &\qquad \times \sup_{y \in B_{\rho/4}(a)}\sup_{t\leq \overline{T}} \TrProb_y (Y_t^\sigma \notin B_{\rho/2}(a), A)\\
    & \quad + \sup_{n \leq N(\sigma)}  \Prob(Y^\sigma_{T_{\mathsf{st}}^\sigma(\kappa) + n \overline{T}} \notin B_{\rho/4}(a), A) \\
    &\qquad \times \sup_{y \in \cDc_r^\mathsf{e}} \sup_{t\leq \overline{T}} \TrProb_y (Y_t^\sigma \notin B_{\rho/2}(a), A).
    \end{aligned}
\end{equation*}

Let us use \eqref{eq:aux:Y_T_st_n_T_1} and bound by $1$ the probabilities that are not needed for our derivations. Finally, we get for any $\kappa > 0$ and $\sigma > 0$ small enough:
\begin{equation*}
    \sup_t \Prob (Y_t^\sigma \notin B_{\rho/2}(a), A) \leq \sup_{y \in B_{\rho/4}(a)}\sup_{t\leq \overline{T}} \TrProb_y (Y_t^\sigma \notin B_{\rho/2}(a)) + \frac{\overline{o}_\kappa(\sigma)}{1 - \overline{o}_\kappa(\sigma)}.   
\end{equation*}

Note that $\{Y_t^\sigma \notin B_{\rho/2}(a)\} \subseteq \{\tau^{Y, \sigma}_{B_{\rho/2}(a)} < t\}$. Therefore, we have
\begin{equation*}
    \sup_{y \in B_{\rho/4}(a)}\sup_{t\leq \overline{T}} \TrProb_y (Y_t^\sigma \notin B_{\rho/2}(a)) \leq \sup_{y \in B_{\rho/4}(a)} \TrProb_y(\tau^{Y, \sigma}_{B_{\rho/2}(a)} < \overline{T}) \xrightarrow[\sigma \to 0]{} 0,
\end{equation*}
by Proposition~\ref{prop:DZ_exit_time}. This finally shows that we can find $\eta > 0$ such that for any $\kappa > 0$ small enough, we have
\begin{equation*}
    \sup_{t \in \left[T_{\mathsf{st}}^\sigma(\kappa); \exp{\frac{2(H + \eta)}{\sigma^2}} \right]} \Prob (Y_t^\sigma \notin B_{\rho/2}(a)) \xrightarrow[\sigma \to 0]{} 0,
\end{equation*}
which proves the lemma.

\subsection{Control of the law: Proof of Lemma~\ref{lem:S_k_control}}
\label{hossein2}

In this paragraph, we prove Lemma~\ref{lem:S_k_control}. In order to do that, we first provide and prove Lemma \ref{lem:trick} below, that is a modification of a technique introduced by J.~Tugaut in~\cite{JOTP2}. Let $\xi(t) := \Wass_2^2(\mu_t^\sigma; \delta_a)$. Consider the following lemma:

\begin{lem}\label{lem:trick}
    Under Assumptions~\ref{assu:pot:V}--\ref{assu:stable}, there exist $K_1, K_2 >0$ such that for any $t>0$:
    \begin{equation*}
        \xi'(t)\le - K_1 \xi(t)+d\sigma^2 + K_2 \sqrt{\Prob(X_t^\sigma \notin B_\rho(a))}\,.
    \end{equation*}
\end{lem}

\begin{proof}
The proof is similar to the one of \cite[Lemma 4.1]{JOTP2} although it is strongly different.

\noindent\noindent{}{\bf Step 1.} First of all, by It\^o's formula, we have
\begin{align*}
    | X_t^\sigma - a|^2 &= | X_0 - a|^2 + 2\sigma\int_0^t \langle X_s^\sigma - a; \dd{B_s} \rangle - 2\int_0^t \langle X_s^\sigma - a; \nabla V(X_s^\sigma)\rangle \dd{s}\\
    & - 2\int_0^t \langle X_s^\sigma - a; \nabla F\ast\mu_s^\sigma (X_s^\sigma)\rangle \dd{s} + d \sigma^2 t.
\end{align*}

For the next step, we take the expectation and derivative with respect to $t$. We get:

\begin{equation*}
    \xi^\prime(t) = d\sigma^2 - 2\E[\langle X_t^\sigma - a; \nabla V(X_t^\sigma) + \nabla F \ast \mu_t^\sigma(X_t^\sigma)\rangle]\,.
\end{equation*}

\noindent\noindent{}{\bf Step 2.} 
Let us introduce $\tilde{F} \in \mathcal{C}^2(\R^d)$ -- a modification of the function $F$ such that $\tilde{F}$ is "convex enough" around $0$. Namely, if $\nabla^2 F(0) \succeq \frac{C_W}{2}{\rm Id}$, where $C_W$ is the positive constant from Definition~\ref{def:Loc_convexity}, then we simply let $\tilde{F} = F$. If not, we introduce a matrix $\cMc := -\nabla^2 F(0) + \frac{C_W}{2}{\rm Id}$ and define $\tilde{F}(x):= F(x) + \frac{1}{2}\left\langle x;\cMc x \right\rangle$.

In the following, without loss of generality, we consider the case $\nabla^2 F(0) \prec \frac{C_W}{2}{\rm Id}$. Moreover, without loss of generality, we assume that $\tilde{F}$ is locally convex inside the ball $B_\rho(0)$, where $\rho$ is the radius of convexity of the effective potential introduced in Definition~\ref{def:Loc_convexity}. Indeed, since $\nabla^2 F$ is continuous, we can always choose $\rho$ in Definition~\ref{def:Loc_convexity} to be small enough such that $\nabla^2 F(x) - \nabla^2 F(0) \succ -\frac{C_W}{2} {\rm Id}$ for any $x \in B_\rho(0)$. Note, that, under these assumptions, $\cMc$ is a positive definite matrix.

\noindent\noindent{}{\bf Step 3.} By definition of $\tilde{F}$, we have:
\begin{align*}
    \E[\langle X^\sigma_t - a&; \nabla F \ast \mu_t^\sigma(X^\sigma_t)\rangle]\\
    & = \E[\langle X^\sigma_t - a; \nabla \tilde{F} \ast \mu_t^\sigma (X^\sigma_t)\rangle] - \E[\langle X^\sigma_t - a; \cMc(X^\sigma_t - \E[X^\sigma_t]) \rangle]\\
    & = \E[\langle X^\sigma_t - a; \nabla \tilde{F} \ast \mu_t^\sigma (X^\sigma_t) \rangle] - \E[\langle X^\sigma_t - a; \cMc(X^\sigma_t - a)\rangle]\\
    & \quad - \E[\langle X^\sigma_t - a; \cMc(a - \E[X^\sigma_t]) \rangle].
\end{align*}

Let $Y_t^\sigma$ be an independent copy of $X_t^\sigma$. Since $\cMc$ is positive definite, this gives us the following lower bound:
\begin{equation}\label{eq:aux:proof_of_trick_1}
\begin{aligned}
    \E[\langle X^\sigma_t - a; \nabla F \ast \mu_t^\sigma(X^\sigma_t)\rangle] &\geq \E\left[\langle X^\sigma_t - a; \nabla \tilde{F}(X^\sigma_t - Y^\sigma_t) \rangle\right] \\
    & \quad - \E[\langle X^\sigma_t - a; \cMc(X^\sigma_t - a)\rangle].
\end{aligned}
\end{equation}

\noindent\noindent{}{\bf Step 4.} We now focus on the first term of the inequality above and. Let us consider separately the parts of the process $X^\sigma$ lying outside and inside the ball $B_{\rho/2}(a)$. Using the polynomial growth (Assumption~\ref{assu:pot:F}), for some generic constant $\text{C}$, we get:

\begin{align*}
    \E\Big[\langle X^\sigma_t-a ;\nabla \tilde{F}(X^\sigma_t-Y^\sigma_t)\rangle\Big] &\geq \E\left[\langle X^\sigma_t-a;\nabla \tilde{F}(X^\sigma_t-Y^\sigma_t)\rangle\mathds{1}_{X^\sigma_t\in B_{\rho/2} (a)} \mathds{1}_{Y^\sigma_t \in B_{\rho/2}(a)} \right] \\
    &\quad - \text{C} \; \E\left[(1 + |X^\sigma_t|^{2r})\mathds{1}_{X^\sigma_t \notin B_{\rho/2}(a)} \right]\,.
\end{align*}

Since $\tilde{F}$ is convex inside $B_\rho(0)$ and the moments are uniformly bounded (Proposition~\ref{prop:existence}), by using the Cauchy–Schwarz inequality, for any $0 < \sigma < 1$, we immediately obtain the existence of a positive constant $K > 0$ such that:
\begin{equation*}
    \E\Big[\langle X^\sigma_t-a ;\nabla \tilde{F}(X^\sigma_t-Y^\sigma_t)\rangle\Big] \geq - K\sqrt{\Prob\left(X^\sigma_t \notin B_{\rho/2}(a)\right)}.
\end{equation*}

We plug this inequality in Equation~\ref{eq:aux:proof_of_trick_1} and get:
\begin{equation}\label{eq:aux:proof_of_trick_2}
\begin{aligned}
    \E[\langle X^\sigma_t - a; \nabla F\ast\mu_t^\sigma (X^\sigma_t)\rangle] &\geq - \E[\langle X^\sigma_t-a;\cMc(X^\sigma_t-a)\rangle]\\
    & \quad - K\sqrt{\Prob\left(X^\sigma_t \notin B_{\rho/2}(a)\right)}.
\end{aligned}
\end{equation}

\noindent\noindent{}{\bf Step 5.} We now focus on the term involving $\nabla V$. According to Definition~\ref{def:Loc_convexity}, for any $x \in B_\rho(a)$, we have:
\begin{equation*}
    \nabla^2 V(x) \succeq C_W {\rm Id} - \nabla^2 F(x - a).
\end{equation*}
At the same time, by the definition of $\mathcal{M}$:
\begin{equation*}
    \nabla^2 F(0) = - \mathcal{M} + \frac{C_W}{2}{\rm Id}.
\end{equation*}
Since $\nabla^2 F$ is continuous, we can, without loss of generality, decrease $\rho$ if necessary so that $- \nabla^2 F(x) \succeq - \nabla^2 F(0) -\frac{C_W}{4}{\rm Id}$ for any $x \in B_\rho(0)$. Therefore, for any $x \in B_\rho(a)$, we have:
\begin{equation*}
    \nabla^2 V(x) \succeq \frac{C_W}{4} {\rm Id} + \mathcal{M}.
\end{equation*}
Using the same logic as in {\bf Step 4}, we get:
\begin{align}
    \E[ \langle &X^\sigma_t-a; \nabla V(X^\sigma_t)\rangle] \nonumber\\
    & = \E[ \langle X^\sigma_t-a; \nabla V(X^\sigma_t)\rangle\mathds{1}_{X^\sigma_t \in B_{\rho/2}(a)}]+\E[ \langle X^\sigma_t - a; \nabla V(X^\sigma_t) \rangle\mathds{1}_{X^\sigma_t \notin B_{\rho/2}(a)}] \nonumber\\
    & \geq \E\left[\langle X^\sigma_t - a; \left(\frac{C_W}{4}{\rm Id} + \cMc \right)(X^\sigma_t - a)\rangle\right] - K \sqrt{\Prob\left(X^\sigma_t \notin B_{\rho/2}(a) \right)}\,. \label{eq:aux:proof_of_trick_3}
\end{align}

\medskip

\noindent{\bf Final step.} As a consequence, putting \ref{eq:aux:proof_of_trick_2} and \ref{eq:aux:proof_of_trick_3} in the {\bf Step 1}, we get:
\begin{equation*}
    \xi^\prime (t) \leq d\sigma^2 - \frac{C_W}{2} \xi(t) + 2K\sqrt{\Prob\left( X^\sigma_t \notin B_{\rho/2}(a)\right)}\,,
    \end{equation*}

which concludes the proof.

\end{proof}

Now we are ready to prove Lemma~\ref{lem:S_k_control} itself.

\begin{proof}[proof of Lemma~\ref{lem:S_k_control}]
    
In order to prove the lemma, we first use Lemma~\ref{lem:trick}, that is inequality:
    \begin{equation*}
        \xi'(t)\le -2\rho'\, \xi(t)+d\sigma^2+K\sqrt{\Prob(X_t^\sigma \notin B_\rho(a))}\,.
    \end{equation*}
After that, we use Lemma \ref{lem:Y_t_notin_B_rho} along with Proposition \ref{prop:|X_t-Y_t|} in order to show that the term $\Prob(X_t^\sigma \notin B_\rho(a))$ tends to $0$ with $\sigma \to 0$ for any $T_{\mathsf{st}}^\sigma(\kappa) \leq t \leq S_{\mathsf{st}}^\sigma(\kappa) \wedge \exp{\frac{2(H + \eta)}{\sigma^2}}$, which, in its term, means that we can choose $\sigma$ to be small enough such that $\xi(t) \leq \kappa^2$ for all such $t$. Final step is to show that $S_{\mathsf{st}}^\sigma(\kappa)$ can not be less or equal than $\exp{\frac{2(H + \eta)}{\sigma^2}}$ or else we get contradiction between the fact that $\xi(S_{\mathsf{st}}^\sigma(\kappa)) \leq \kappa^2$ and definition of $S_{\mathsf{st}}^\sigma(\kappa)$.

Consider the following inequalities. For any $T_{\mathsf{st}}^\sigma(\kappa) \leq t \leq S_{\mathsf{st}}^\sigma(\kappa) \wedge \exp{\frac{2(H + \eta)}{\sigma^2}}$:
\begin{equation*}
        \Prob(X_t^\sigma \notin B_{\rho}(a)) \leq \Prob(Y^\sigma_t \notin B_{\rho/2}(a)) + \Prob(|Y^\sigma_t - X^\sigma_t| > \rho/2) = o_\sigma(1),
\end{equation*}
by Lemma \ref{lem:Y_t_notin_B_rho} and Proposition \ref{prop:|X_t-Y_t|}. Thus, by Lemma~\ref{lem:trick}, $\xi(t) = \Wass^2_2(\mu_t; \delta_a)$ is bounded for any $t$ considered above in the following way:
\begin{equation*}
    \xi^\prime(t)\le -2\rho^\prime \xi(t)+d\sigma^2 + K o_\sigma(1).
\end{equation*}
Therefore, we can decrease $\kappa$ and then $\sigma$ to be small enough such that $\xi(t) \leq \kappa^2$ for any $T_{\mathsf{st}}^\sigma(\kappa) \leq t \leq S_{\mathsf{st}}^\sigma(\kappa) \wedge \exp{\frac{2(H + \eta)}{\sigma^2}}$. 

The last step is to note that if $S_{\mathsf{st}}^\sigma(\kappa) < \exp{\frac{2(H + \eta)}{\sigma^2}}$, then we get a contradiction between the definition of $S_{\mathsf{st}}^\sigma(\kappa)$ and the fact that $\xi(S_{\mathsf{st}}^\sigma(\kappa)) \leq \kappa^2$, which proves the lemma.

\end{proof}

\begin{small}
\bibliographystyle{plain}
\bibliography{biblioKramer.bib}
\end{small}

\end{document}